\documentclass[11pt,a4paper,reqno]{amsart}

\usepackage{amssymb,amsmath}
\usepackage{amsthm}
\usepackage{latexsym}
\usepackage{wrapfig}
\usepackage[colorlinks=true,linktocpage=true,pagebackref=false,citecolor=black,linkcolor=black]{hyperref}
\usepackage{graphics}
\usepackage{euscript}
\usepackage{mathrsfs}
\usepackage[dvips]{curves}
\usepackage{tikz}
\usetikzlibrary{arrows,decorations.pathmorphing,backgrounds,positioning,fit,matrix}
\usepackage{comment}
\usepackage{multirow}
\usepackage{xcolor}
\usepackage{caption,subcaption}

\def\sqbullet{\raise.2ex\hbox{\vrule width 3.5pt height 3.5pt}}

\newtheorem{thm}{Theorem}[section]

\newtheorem{prop}[thm]{Proposition}
\newtheorem{cor}[thm]{Corollary}
\newtheorem{lem}[thm]{Lemma}

\theoremstyle{definition}

\newtheorem{define}[thm]{Definition}

\theoremstyle{remark}
\newtheorem{remark}[thm]{Remark}
\newtheorem{remarks}[thm]{Remarks}

\numberwithin{equation}{section}
 \newcommand{\R}{{\mathbb R}}

\newcommand{\sph}{{\mathbb S}}

\newcommand{\pol}{{\EuScript K}}

\newcommand{\p}{{\EuScript P}}

\newcommand{\Ss}{{\EuScript S}}

\newcommand{\Qq}{{\EuScript Q}}

\newcommand\Rr{{\EuScript R}}

\newcommand{\Mm}{{\EuScript M}}
\newcommand{\Tt}{{\EuScript T}}

\newcommand{\Ff}{{\EuScript F}}

\newcommand{\Cc}{{\EuScript C}}

\newcommand{\Ee}{{\EuScript E}}
\newcommand{\Gg}{{\EuScript G}}
\newcommand{\tildebaja}{{\raise.17ex\hbox{$\scriptstyle\sim$}}}

\newcommand{\Int}{\operatorname{Int}}
\newcommand{\im}{\operatorname{Im}}
\newcommand{\pp}{\operatorname{p}}
\newcommand{\rr}{\operatorname{r}}

\newcommand{\cl}{\operatorname{Cl}}
\newcommand{\dist}{\operatorname{dist}}
\newcommand{\id}{\operatorname{id}}

\newcommand{\conv}[2]{\vec{{\mathfrak C}}_{#2}(#1)}

\newcommand{\x}{{\tt x}} \newcommand{\y}{{\tt y}}
\newcommand{\z}{{\tt z}} \renewcommand{\t}{{\tt t}}

\newcommand{\ol}{\overline}

\newcommand{\veps}{\varepsilon}

\newcommand{\tth}{{\tt h}}
\newcommand{\ttg}{{\tt h_0}}
\newcommand{\ttgg}{{\tt g}}
\newcommand{\ttl}{{\tt l}}
\newcommand{\tts}{{\tt l_0}}
\newcommand{\ttf}{{\tt f}}
\newcommand{\ttq}{{\tt q}}
\newcommand{\ttL}{{\tt h}}

\numberwithin{equation}{section}
\renewcommand\thesubsection{\thesection.\Alph{subsection}}

\begin{document}

\definecolor{color1}{rgb}{0.20,0.20,0.20}
\definecolor{color2}{rgb}{0.40,0.40,0.40}
\definecolor{color3}{rgb}{0.60,0.60,0.60}

\title[On complements of convex polyhedra as polynomial images of $\R^n$]{On complements of convex polyhedra\\ as polynomial images of $\R^n$}

\author{Jos\'e F. Fernando}
\address{Departamento de \'Algebra, Facultad de Ciencias Matem\'aticas, Universidad Complutense de Madrid, 28040 MADRID (SPAIN)}
\curraddr{Dipartimento di Matematica, Universit\`a degli Studi di Pisa, Largo Bruno Pontecorvo, 5, 56127 PISA (ITALY)}
\email{josefer@mat.ucm.es}

\author{Carlos Ueno}
\address{Dipartimento di Matematica, Universit\`a degli Studi di Pisa, Largo Bruno Pontecorvo, 5, 56127 PISA (ITALY)}
\email{jcueno@mail.dm.unipi.it}

\thanks{Both authors are supported by Spanish GR MTM2011-22435. First author is also supported by Grupos UCM 910444. This work was also partially supported by the ``National Group for Algebraic and Geometric Structures and their Applications'' (GNSAGA - INdAM). This article has been written during a common one year research stay of the authors in the Dipartimento di Matematica of the Universit\`a di Pisa. The authors would like to thank the department for the invitation and the very pleasant working conditions. The one-year research stay of the first author is partially supported by MECD grant PRX14/00016.}

\date{04/05/2015}

\begin{abstract}
In this work we prove constructively that the complement $\R^n\setminus\pol$ of an $n$-dimensional unbounded convex polyhedron $\pol\subset\R^n$ and the complement $\R^n\setminus\Int(\pol)$ of its interior are polynomial images of $\R^n$ whenever $\pol$ does not disconnect $\R^n$. The compact case and the case of convex polyhedra of small dimension were approached by the authors in previous works. Consequently, the results of this article provide a full answer to the representation as polynomial images of Euclidean spaces of complements of convex polyhedra and its interiors. The techniques here are more sophisticated than those corresponding to the compact case and require a rational separation result for certain type of (non-compact) semialgebraic sets, that has interest by its own. 
\end{abstract}

\maketitle
 
\section{Introduction and statement of the main results}\label{s1}

A map $f:=(f_1,\ldots,f_m):\R^n\to\R^m$ is \em polynomial \em if its components $f_k\in\R[\x]:=\R[\x_1,\ldots,\x_n]$ are polynomials. Analogously, $f$ is \em regular \em if its components can be represented as quotients $f_k=\frac{g_k}{h_k}$ of two polynomials $g_k,h_k\in\R[\x]$ such that $h_k$ never vanishes on $\R^n$. A subset $\Ss\subset\R^n$ is \em semialgebraic \em when it has a description by a finite boolean combination of polynomial equations and inequalities. By Tarski-Seidenberg's principle \cite[1.4]{bcr} the image of an either polynomial or regular map is a semialgebraic set. During the last decade we have approached the problem of characterizing which (semialgebraic) subsets $\Ss\subset\R^m$ are polynomial or regular images of $\R^n$. The first proposal of studying this problem and related ones like the famous `quadrant problem' go back to \cite{g}.

The effective representation of a subset $\Ss\subset\R^m$ as a polynomial or regular image of $\R^n$ reduces the study of certain classical problems in Real Geometry to its study in $\R^n$, with the advantage of avoiding contour conditions. Examples of such problems appear in Optimization or in the search for Positivstellens\"atze certificates \cite{fg2,fu2}. Observe that these representations provide Positivstellensatz certificates for non-necessarily closed nor basic semialgebraic sets, whenever we are able to represent them as regular or polynomial images of $\R^n$.

We feel very far from solving the problem stated above in its full generality, but we have developed significant progresses in two ways:

\vspace*{1mm}
\noindent{\em Obtention of general properties.} We have found conditions \cite{f1,fg2,fu2,u1} that a semialgebraic subset $\Ss\subset\R^m$ must satisfy in order to be an either polynomial or regular image of $\R^n$. The most remarkable one states that the set of points at infinity of a polynomial image of $\R^n$ is connected \cite{fu1}. In addition, the $1$-dimensional case has been completely described in \cite{f1}. 

\vspace*{1mm}
\noindent{\em Explicit representation of families of semialgebraic sets as polynomial or regular images of $\R^n$.} We have devised techniques to represent large families of significant semialgebraic sets as either polynomial or regular images of $\R^n$. In \cite{f1,fg1,fgu1,u2} we focused on semialgebraic sets with piecewise linear boundary, that is, semialgebraic sets that admit a semialgebraic description involving only linear equations. To be more precise, we analyzed the cases of convex polyhedra and their interiors, together with their respective complements. For these families of semialgebraic sets we have full information concerning their representation as regular images \cite{fgu1,fu2} but we still lack much information when trying to represent them as polynomial (instead of regular) images.

Let us be more explicit in this point. Since the complements of proper convex polyhedra (or of their interiors) are unbounded semialgebraic sets, it makes sense to wonder whether these sets are not only images of regular maps, but also of polynomial maps. Our initial purpose when writing \cite{fu2} was to approach the previous problem in its full generality, but the techniques developed there required, in order to use polynomial maps, to assume that the involved convex polyhedra were compact when the dimension of $\pol$ matched that of the ambient space and was greater than or equal to $4$ (for further details see \cite{fu3}). Therefore, the main results appearing in \cite{fu2} refer to the compact case, together with the case of convex polyhedra of smaller embedding dimension:

\begin{thm}\label{main-known}\em (\cite[Thm. 1.1(i)]{fu2}) \em
Let $\pol$ be an $n$-dimensional compact convex polyhedron of $\R^n$. Then the semialgebraic sets $\Ss:=\R^n\setminus\pol$ and $\ol{\Ss}:=\R^n\setminus\Int(\pol)$ are polynomial images of $\R^n$.
\end{thm}

\begin{prop}\em (\cite[Thm. 3.1]{fu2}) \em
Let $\pol$ be an $d$-dimensional convex polyhedron of $\R^n$ such that $d<n$ and it is not a hyperplane. Then the semialgebraic sets $\Ss:=\R^n\setminus\pol$ and $\ol{\Ss}:=\R^n\setminus\Int(\pol)$ are polynomial images of $\R^n$.
\end{prop}

Similar techniques to those developed to prove Theorem \ref{main-known} can be adapted for unbounded $2$-dimensional and $3$-dimensional convex polyhedra \cite{fu3,u2}, but unfortunately they do not extend any further to higher dimensions \cite{fu3}. The purpose of this work is to close this gap and provide a full answer to the representation of the complements of convex polyhedra $\pol$ and their interiors as polynomial images of $\R^n$, disposing of the compactness assumption concerning $\pol$ that appears in \cite{fg2}. Since our previous methods did not work in this more general setting, we have developed new tools that use more sophisticated polynomial maps to achieve our goals. This requires a more technical approach than the one devised in \cite{fg2}, but reveals a better understanding on how polynomial maps can act on $\R^n$ to produce our desired image sets.

\subsection{Main results} 
From now on, we denote by $\Int(\pol)$ the relative interior of $\pol$ as a topological manifold with boundary. A \em layer \em is a convex polyhedron of $\R^n$ affinely equivalent to $[-a,a]\times\R^{n-1}$ with $a>0$. Our main results in this work, which complete the full picture in regard to the representation of complements of convex polyhedra and their interiors as polynomial images of Euclidean spaces, are the following:

\begin{thm}\label{main1}
Let $n\geq1$ and let $\pol$ be an $n$-dimensional unbounded convex polyhedron in $\R^n$ that is not a layer. Then the semialgebraic set $\ol{\Ss}:=\R^n\setminus\Int(\pol)$ is a polynomial image of $\R^n$.
\end{thm}

\begin{thm}\label{main2}
Let $n\geq2$ and let $\pol$ be an $n$-dimensional unbounded convex polyhedron in $\R^n$ that is not a layer. Then the semialgebraic set $\Ss:=\R^n\setminus\pol$ is a polynomial image of $\R^n$.
\end{thm}

It is worthwhile to mention here that the proof of Theorem \ref{main2} involves a separation result for tuples of variables that has interest by its own. A \em rational separator for the pair of positive integers $(r,s)$ \em is a rational function $\phi_{r,s}:\R^r\times\R^s\dashrightarrow\R$ that is regular on the interior of the polyhedron
$$
\Qq_{r,s}:=\{(y_1,\dots,y_r;z_1,\dots,z_s)\in\R^r\times\R^s:\ \max\{y_1,\dots,y_r\}\le \min\{z_1,\dots,z_s\}\},
$$ 
extends to a continuous (semialgebraic) function on $\Qq_{r,s}$ and satisfies
$$
\max\{y_1,\ldots,y_r\}<\phi_{r,s}(y;z)<\min\{z_1,\ldots,z_s\}
$$
for each $(y;z):=(y_1,\dots,y_r;z_1,\dots,z_s)\in\Int(\Qq_{r,s})$. In Lemma \ref{lemma} we show that for each pair of positive integers $(r,s)$ there exists rational separators. As a consequence we prove in Lemma \ref{prop:polP} that given an $n$-dimensional convex polyhedron of $\R^n$ and the projection $\pi_n:\R^{n}\to\R^{n-1},\ (x_1,\ldots,x_n)\to(x_1,\ldots,x_{n-1})$, the two connected components of the difference $(\Int(\pi_n(\pol))\times\R)\setminus\pol$ can be separated by a rational function that is regular on $\Int(\pi_n(\pol))$, extends to a continuous function on $\pi_n(\pol)$ and has a pole set contained in its zero set. As it is well-known, the separation of disjoint semialgebraic sets is a delicate issue and we refer the reader to \cite{aab} for further details.

\subsection{Full picture} To ease the presentation of the full picture of what is known concerning piecewise linear boundary semialgebraic sets \cite{fgu1,fu1,fu2,u2} we introduce the following two invariants. Given a semialgebraic set $\Ss\subset\R^m$, we define
\begin{equation*}
\begin{split}
\pp(\Ss):&=\inf\{n\geq1:\exists \ f:\R^n\to\R^m\, \, \, \text{polynomial such that}\, \, \, f(\R^n)=\Ss\},\\
\rr(\Ss):&=\inf\{n\geq1:\exists \ f:\R^n\to\R^m\, \, \, \text{regular such that}\, \, \, f(\R^n)=\Ss\}.
\end{split}
\end{equation*}
The condition $\pp(\Ss):=+\infty$ characterizes the non representability of $\Ss$ as a polynomial image of some $\R^n$ while $\rr(\Ss):=+\infty$ has the analogous meaning for regular maps. Let $\pol\subset\R^n$ be an $n$-dimensional convex polyhedron and assume below that $\Ss:=\R^n\setminus\pol$ and $\ol{\Ss}=\R^n\setminus\Int(\pol)$ are connected. 
\begin{table}[ht]
$$
\renewcommand*{\arraystretch}{1.3}
\begin{array}{|c|c|c|c|c|c|c|c|c|c|}
\hline 
&\multicolumn{2}{c|}{\text{$\pol$ bounded}}&\multicolumn{4}{c|}{\text{$\pol$ unbounded}}\\[2pt] 
\cline{2-7} 
&n=1&n\geq2&n=1&\multicolumn{3}{c|}{n\geq2}\\[2pt] 
\hline
{\rm r}(\pol)&1&\multirow{2}{*}{$n$}&1&\multicolumn{3}{c|}{\multirow{2}{*}{$n$}}\\[2pt] 
\cline{1-2}\cline{4-4}
{\rm r}(\Int(\pol))&2&&2&\multicolumn{3}{c|}{}\\[2pt] 
\cline{1-1}\cline{2-7}
{\rm p}(\pol)&\multicolumn{2}{c|}{\multirow{2}{*}{$+\infty$}}&1&\multicolumn{3}{c|}{n, +\infty\ (\ast)}\\[2pt] 
\cline{1-1}\cline{4-7}
{\rm p}(\Int(\pol))&\multicolumn{2}{c|}{}&2&\multicolumn{3}{c|}{n, n+1, +\infty\ (\star)}\\[2pt] 
\cline{1-7}
{\rm r}(\Ss)&\multirow{4}{*}{$+\infty$}&\multirow{4}{*}{$n$}&2&\multicolumn{3}{c|}{\multirow{4}{*}{$n$}}\\[2pt] 
\cline{1-1}\cline{4-4}
{\rm r}(\ol{\Ss})&&&1&\multicolumn{3}{c|}{}\\[2pt] 
\cline{1-1}\cline{4-4}
{\rm p}(\Ss)&&&2&\multicolumn{3}{c|}{}\\[2pt] 
\cline{1-1}\cline{4-4} 
{\rm p}(\ol{\Ss})&&&1&\multicolumn{3}{c|}{}\\[2pt] 
\hline
\end{array}
$$
\caption{State of the art}\label{tabla}
\end{table}

\vspace{-10mm}
Let us explain some (marked) cases in {\sc Table}~\ref{tabla} that are being developed in \cite{fgu2}:
\begin{itemize}
\item[$(\ast)$] $(n,\,+\infty)$: An $n$-dimensional convex polyhedron $\pol\subset\R^n$ has ${\rm p}(\pol)=+\infty$ if and only if its recession cone $\conv{\pol}{}$ has dimension $<n$; otherwise ${\rm p}(\pol)=n$.
\item[$(\star)$] $(n,n+1,\,+\infty)$: If the recession cone $\conv{\pol}{}$ of an $n$-dimensional convex polyhedron $\pol$ has dimension $<n$, then ${\rm p}(\Int(\pol))=+\infty$. Otherwise, if $\pol$ has bounded facets, ${\rm p}(\Int(\pol))=n+1$ and if $\pol$ has no bounded facets, ${\rm p}(\Int(\pol))=n$.
\end{itemize}

\subsection{Structure of the article} All basic notions and (standard) notation appear in Section \ref{s2}. The reader can start directly in Section \ref{s3} and return to these preliminaries when needed. In Section \ref{s3} we prove Theorem~\ref{main1}, while Theorem~\ref{main2} is proved in Section~\ref{s5}. In Section~\ref{s4} we provide some rational separation results for certain types of (non-compact) semialgebraic sets. These results are the key to construct explicitly a polynomial map whose image is the complement of a convex polyhedron.

\section{Preliminaries on convex polyhedra}\label{s2}

We begin by introducing some preliminary terminology and notations concerning convex polyhedra. For a detailed study of the main properties of convex sets we refer the reader to \cite{ber1,r,z}. An affine hyperplane of $\R^n$ will be written as $H:=\{x\in\R^n:\ {\tth}(x)=0\}\equiv\{{\tth}=0\}$ for a linear equation ${\tth}$. It determines two \emph{closed half-spaces}
$$
H^+:=\{x\in\R^n:\ {\tth}(x)\geq 0\}\equiv\{{\tth}\geq0\}\quad\text{and}\quad H^-:=\{x\in\R^n:\ {\tth}(x)\leq 0\}\equiv\{{\tth}\leq0\}.
$$
We use an overlying arrow $\vec{\cdot}$ when referring to the corresponding vectorial objects. An affine subspace $W$ of $\R^n$ is called \em vertical \em if it is parallel to the vector $\vec{e}_n:=(0,\ldots,0,1)$. Otherwise, we say that $W$ is \em non-vertical\em.

\subsection{Generalities on convex polyhedra}\label{pldrfc} 
A subset $\pol\subset\R^n$ is a \emph{convex polyhedron} if it can be described as the finite intersection $\pol:=\bigcap_{i=1}^rH_i^+$ of closed half-spaces $H_i^+$. The dimension $\dim(\pol)$ of $\pol$ is the dimension of the smallest affine subspace of $\R^n$ that contains $\pol$. If $\pol$ has non-empty interior there exists by \cite[12.1.5]{ber1} a unique minimal family $\{H_1,\ldots,H_m\}$ of affine hyperplanes in $\R^n$ such that $\pol=\bigcap_{i=1}^mH_i^+$. This family is the \em minimal presentation \em of $\pol$. We assume that we choose the linear equation ${\tth}_i$ of each $H_i$ so that $\pol\subset H_i^+$. For inductive processes we will write $\pol_{i,\times}:=\bigcap_{j\neq i}H_j^+$, which is a convex polyhedron that strictly contains $\pol$, satisfies $\pol=\pol_{i,\times}\cap H_i^+$ and has one facet less than $\pol$. 

\subsubsection{}The \em facets \em or $(n-1)$-\em faces \em of $\pol$ are the intersections $\Ff_i:=H_i\cap\pol$ for $1\leq i\leq m$. Only the convex polyhedron $\R^n$ has no facets. Each facet $\Ff_i:=H_i^-\cap\bigcap_{j=1}^mH_j^+$ is a convex polyhedron contained in $H_i$. The convex polyhedron $\pol\subset\R^n$ is a topological manifold with boundary, whose interior is $\Int\pol=\bigcap_{i=1}^m(H_i^+\setminus H_i)$ and its boundary is $\partial\pol=\bigcup_{i=1}^m\Ff_i$. For $0\leq j\leq n-2$ we define inductively the $j$-\em faces \em of $\pol$ as the facets of the $(j+1)$-faces of $\pol$, which are again convex polyhedra. The $0$-faces are the \em vertices \em of $\pol$ and the $1$-faces are the \em edges \em of $\pol$. A face $\Ee$ of $\pol$ is \em vertical \em if the affine subspace of $\R^n$ generated by $\Ee$ is vertical. Otherwise, we say that $\Ee$ is \em non-vertical\em. Obviously, if $\pol$ has a vertex, then $m\geq n$. A convex polyhedron of $\R^n$ is \em non-degenerate \em if it has at least one vertex. Otherwise, we say that the convex polyhedron is \em degenerate\em. 

\subsubsection{}A {\em supporting hyperplane} of a convex polyhedron $\pol\subset\R^n$ is a hyperplane $H$ of $\R^n$ that intersects $\pol$ and satisfies $\pol\subset H^+$ or $\pol\subset H^-$. This is equivalent to have $\varnothing\neq\pol\cap H\subset\partial\pol$. The intersection of $\pol$ with a supporting hyperplane $H$ is a face of $\pol$ and conversely each face of $\pol$ is the intersection of $\pol$ with some supporting hyperplane. In particular, the vertices of a convex polyhedron $\pol\subset\R^n$ are those points $p\in\pol$ for which there exists a (supporting) hyperplane $H\subset\R^n$ such that $\pol\cap H=\{p\}$. 

\subsection{Projections of convex polyhedra}\label{lem:basic}
Let $\pol\subset\R^n$ be an $n$-dimensional convex polyhedron and let $\pi_n:\R^n\to\R^{n-1},\ x:=(x',x_n)\to x'$ be the projection onto the first $(n-1)$ coordinates. We denote the origin of $\R^n$ with ${\bf0}$ and that of $\R^{n-1}$ with ${\bf0}'$. Let $\p:=\pi_n(\pol)$ and let $\vec{\ell}_n$ be the line generated by $\vec{e}_n:=(0,\ldots,0,1)=({\bf0}',1)$. By \cite[II.Thm.6.6]{r} we have $\pi_n(\Int(\pol))=\Int(\p)$ and consequently $\pi_n^{-1}(\partial\p)\cap\pol\subset\partial\pol$. Thus, if $\ell$ is a vertical line and $\pi_n(\ell)\subset\partial\p$, then
$$
\ell\cap\pol=\pi_n^{-1}(\pi_n(\ell))\cap\pol\subset\pi_n^{-1}(\partial\p)\cap\pol\subset\partial\pol.
$$
In fact, \em $\pi_n^{-1}(\partial\p)\cap\pol\subset\partial\pol$ is a union of faces of $\pol$\em.

Indeed, let $\Ff'$ be a facet of $\p$ and let $H'$ be the hyperplane of $\R^{n-1}$ generated by $\Ff'$. Notice that $H:=\pi^{-1}(H')=(H'\times\{0\})+\vec{\ell}_n$ is a hyperplane of $\R^n$ that meets $\pol$ but does not meet $\Int(\pol)$. Thus, $H$ is a supporting hyperplane of $\pol$ and $H\cap\pol=:\Ee$ is a face of $\pol$. Therefore
$$
\pi^{-1}(\Ff')\cap\pol=\pi^{-1}(H'\cap\p)\cap\pol=\pi^{-1}(H')\cap\pi^{-1}(\p)\cap\pol=H\cap\pol=\Ee
$$
is a face of $\pol$.

\subsection{Recession cone of a convex polyhedron}

We associate to each convex polyhedron $\pol\subset\R^n$ its \em recession cone\em, see \cite[Ch.1]{z} and \cite[II.\S8]{r}. Fix a point $p\in\pol$vand consider for each vector $\vec{v}\in\R^n$ the ray $p\vec{v}$ with origin at $p$ and direction $\vec{v}$. \em Then the set $\conv{\pol}{}:=\{\vec{v}\in\R^n:\,p\vec{v}\subset\pol\}$ is a convex cone and it does not depend on the choice of $p$\em. This set $\conv{\pol}{}$ is called the \em recession cone \em of $\pol$. If $\pol:=\bigcap_{i=1}^rH_i^+$, then $\conv{\pol}{}:=\bigcap_{i=1}^r\conv{H_i^+}{}=\bigcap_{i=1}^r\vec{H_i}^+$. Clearly, $\conv{\pol}{}=\{{\bf0}\}$ if and only if $\pol$ is bounded. In addition, if $\p\subset\R^n$ is a non-degenerate convex polyhedron and $k\geq1$, then $\conv{\R^k\times\p}{}=\R^k\times\conv{\p}{}$. Recall that each degenerate convex polyhedra can be written as the product of a non-degenerate convex polyhedron times an Euclidean space. Besides, a convex polyhedron is degenerate if and only if it contains a line or, equivalently, if its recession cone contains a line. Consequently a convex polyhedron is non-degenerate if and only if all its faces are non-degenerate polyhedra.

\section{Complements of interiors of convex polyhedra}\label{s3}

The purpose of this section is to provide a constructive proof of Theorem \ref{main1}. The proof can be schematized as follows: We proceed by double induction on the dimension and the number of facets of $\pol$. We place the polyhedron $\pol$ so that the vector $-\vec{e}_n=(0,\dots,0,-1)$ lies in its recession cone and one of its facets $\Ff_i$ lies in the hyperplane $\{\x_n=0\}$. By the induction hypothesis, the complement of the interior of the polyhedron $\pol_{i,\times}$ (which has one facet less than $\pol$) is a polynomial image of $\R^n$. The main task is to construct a polynomial map ${\tt F}$ that sends this complement onto the complement of the original polyhedron $\pol$. The right image in Figure~\ref{im:poly} describes the behavior of the polynomial map we are about to construct. In order to clarify this process we split it in two stages. In the first one we reduce the proof of Theorem \ref{main1} to show that the complement in $\R^2$ of the interior of a convex polygon is the image of certain polynomial maps ${\ttf}:\R^2\to\R^2$. In the second stage we actually prove this fact.

\subsection{Reduction of the proof of Theorem \ref{main1} to the 2-dimensional case}\label{subs:red}
The reduction is conducted in several steps.

\subsubsection{Setting up the scenario}
We proceed by double induction on the pair $(n,m)$, where $n$ denotes the dimension of $\pol$ and $m$ its number of facets. The result is trivial for $n=1$ because in this case layers correspond precisely to bounded closed intervals that disconnect $\R$, whereas $\R^n\setminus\Int(\pol)$ for unbounded $\pol$ is affinely equivalent to ${[0,+\infty[}$, which is the image of the polynomial map ${\ttf}:\R\to\R,\ x\mapsto x^2$. Assume $n\geq2$ and the result true for all unbounded convex polyhedra that have either dimension $\leq n-1$, or dimension $n$ and less than $m$ facets.

Let $\pol\subset\R^n$ be an $n$-dimensional unbounded convex polyhedron with $m$ facets, which is not a layer. If $\pol$ is degenerate, we can assume $\pol=\p\times\R$ where $\p\subset\R^{n-1}$ is a convex polyhedron different from a layer. By the induction hypothesis there exists a polynomial map ${\ttf}_0:\R^{n-1}\to\R^{n-1}$ whose image is $\R^{n-1}\setminus\Int(\p)$. The image of the polynomial map ${\ttf}:\R^n\to\R^n,\ x:=(x',x_n)\mapsto({\ttf}_0(x'),x_n)$ is $\R^n\setminus\Int(\pol)$ and we are done.

Assume that $\pol$ is non-degenerate and unbounded, with $-\vec{e}_n\in{\rm Int}(\conv{\pol}{})$. Let $\Ff_1,\ldots,\Ff_r$ be the vertical facets of $\pol$ and let $\Ff_{r+1},\ldots,\Ff_{m}$ be the non-vertical ones. As $\pol$ is non-degenerate, $r<m$. Depending on the value of $r$ we distinguish two cases:

\subsubsection{\sc Case 1} If $r<m-1$, assume, after a change of coordinates that keeps the vector $\vec{e}_n$ invariant, that $\Ff_m\subset\{\x_n=0\}$ and the origin is contained in the interior of $\Ff_m$. In particular we have $\pol\subset\{\x_n\le 0\}$.

\subsubsection{\sc Case 2} If $r=m-1$, we need an intermediate polynomial map that transforms $\R^n\setminus\pol_{m,\times}$ into a semialgebraic set more suitable to our purposes. The convex polyhedron $\pol_{m,\times}:=H_1^+\cap\cdots\cap H_{m-1}^+$ is degenerate. All the (unbounded) faces of $\pol_{m,\times}$ are vertical, so $\pol_{m,\times}=\p\times\R$ where $\p$ is non-degenerate (because otherwise $\pol$ would be degenerate).

\paragraph{} Let us check: \em $\p$ is bounded\em. 

As $H_m$ is non-vertical, after a change of coordinates that keeps $\vec{e}_n$ invariant we may assume that $H_{m}=\{-\x_n=0\}$ and also that the origin of $\R^{n-1}$ is contained in $\Int(\p)$. Since $\pol=\pol_{m,\times}\cap H_m^+$ and $\pol_{m,\times}=\p\times\R$, we have $\conv{\pol}{}=\conv{\pol_{m,\times}}{}\cap\vec{H}_m^+$ and $\conv{\pol_{m,\times}}{}=\conv{\p}{}\times\R$. Thus, 
$$
\conv{\pol}{}=(\conv{\p}{}\times\R)\cap\vec{H}_m^+=(\conv{\p}{}\times\R)\cap\{-\x_n\geq0\}=\conv{\p}{}\times{]{-\infty},0]}
$$ 
and consequently
$$
\partial\conv{\pol}{}=(\partial\conv{\p}{}\times{]-\infty,0]})\cup(\conv{\p}{}\times\{0\}).
$$
As $\p$ is non-degenerate, then either ${\bf0}'\in\partial\conv{\p}{}$ or $\conv{\p}{}=\{{\bf0'}\}$. As $-\vec{e}_n=({\bf0}',-1)\in{\rm Int}(\conv{\pol}{})$, we conclude $\conv{\p}{}=\{{\bf0'}\}$, so $\p$ is bounded. 

\paragraph{}Assume $\pol:=\p\times{]{-\infty},0]}$ where $\p$ is a bounded convex polyhedron and the origin of $\R^{n-1}$ is contained in $\Int(\p)$. Write $H_i:=\{{\tth}_i=0\}$ where ${\tth}_i$ is a linear equation and let ${\ttL}:=\prod_{i=1}^{m-1}{\tth}_i$. Notice that ${\ttL}$ does not depend on the variable $\x_n$. We claim: \em The polynomial map
$$
{\tt F}_0(\x_1,\dots,\x_n):=((1-\x_n{\ttL}^2(\x_1,\ldots,\x_{n-1}))\x_1,\dots,(1-\x_n{\ttL}^2(\x_1,\ldots,\x_{n-1}))\x_{n-1},\x_n)
$$
satisfies ${\tt F}_0(\R^n\setminus\Int(\p\times\R))=\R^n\setminus(\Int(\p)\times{]{-\infty},0]})$.
\em

Indeed, given $\vec{u}\in\sph^{n-2}\subset\R^{n-1}$ denote the line $\{t\vec{u}:\ t\in\R\}$ generated by the vector $\vec{u}$ with $\ell_{\vec{u}}$. As $\Int(\p\times\R)=\Int(\p)\times\R$, 
\begin{multline*}
\R^n=\bigcup_{\lambda\in\R}\bigcup_{\vec{u}\in\sph^{n-2}}\ell_{\vec{u}}\times\{\lambda\},\quad \R^n\setminus\Int(\p\times\R)=\bigcup_{\lambda\in\R}\bigcup_{\vec{u}\in\sph^{n-2}}(\ell_{\vec{u}}\setminus\Int(\p))\times\{\lambda\},\\
\R^n\setminus(\Int(\p)\times{]{-\infty},0]})=\bigcup_{\lambda\leq0}\bigcup_{\vec{u}\in\sph^{n-2}}(\ell_{\vec{u}}\setminus\Int(\p))\times\{\lambda\}\cup\bigcup_{\lambda>0}\bigcup_{\vec{u}\in\sph^{n-2}}\ell_{\vec{u}}\times\{\lambda\}.
\end{multline*}
We must show that
$$
{\tt F}_0((\ell_{\vec{u}}\setminus\Int(\p))\times\{\lambda\})=\begin{cases}
(\ell_{\vec{u}}\setminus\Int(\p))\times\{\lambda\}&\text{if $\lambda\leq0$},\\
\ell_{\vec{u}}\times\{\lambda\}&\text{if $\lambda>0$}.\\
\end{cases}
$$
Fix $\vec{u}\in\sph^{n-2}$ and let $a<0<b$ be such that $\ell_{\vec{u}}\cap\p=\{t\vec{u}:\ t\in[a,b]\}$ (recall that the origin of $\R^{n-1}$ is contained in $\Int(\p)$). Fix $\lambda\in\R$ and write ${\tt F}_0({\tt t}\vec{u},\lambda)=({\tt P}_{\vec{u},\lambda}({\tt t})\vec{u},\lambda)$ where ${\tt P}_{\vec{u},\lambda}({\tt t}):=(1-\lambda{\ttL}^2({\tt t}\vec{u})){\tt t}$. We must show that
$$
{\tt P}_{\vec{u},\lambda}(\R\setminus]a,b[)=\begin{cases}
\R\setminus]a,b[&\text{if $\lambda\leq0$,}\\
\R&\text{if $\lambda>0$.}
\end{cases}
$$
Figure \ref{fig:graphs} shows the possible behaviors of ${\tt P}_{\vec{u},\lambda}$. Observe that ${\tt P}_{\vec{u},\lambda}(a)=a$ and ${\tt P}_{\vec{u},\lambda}(b)=b$ because ${\ttL}$ is identically zero on the boundary of $\p$. If $\lambda\leq0$, we have $(1-\lambda{\ttL}^2(t\vec{u}))\geq1$ for all $t\in\R$, so
$$
{\tt P}_{\vec{u},\lambda}(t)\begin{cases}
\geq t&\text{if $t\geq0$,}\\
\leq t&\text{if $t<0$.}
\end{cases}
$$

\begin{figure}[ht]
\begin{center}
\begin{tikzpicture}[scale=1]
{\small
\draw[<->] (2.5,0) -- (2.5,5.5);
\draw[<->] (0,2.5) -- (5.5,2.5);
\draw[dashed] (0,0) to (5.5,5.5);
\draw[line width=1pt] (0.75,0) .. controls (1.2,1.25) and (1.4,1.4) .. (1.5,1.5);
\draw[line width=1pt] (4,4) .. controls (4.2,4.25) and (4.4,4.45) .. (4.5,5.5);
\draw[line width=1pt, dashed] (1.5,1.5) .. controls (1.75,1.75) and (2.25,1.7) .. (2.5,1.625) .. controls (2.75,1.5) and (3.4,1.75) .. (3.5,2.5) .. controls (3.7,3.75) and (3.95,4) .. (4,4);
\draw (1.5,1.5) node{$\bullet$};
\draw (4,4) node{$\bullet$};
\draw (5,4.3) node{${\tt z}={\tt t}$};
\draw (4,1.5) node{${\tt z}={\tt P}_{\vec{u},\lambda}({\tt t})$};
\draw (1.5,2.5) node{\tiny$|$};
\draw (4,2.5) node{\tiny$|$};
\draw (2.5,1.5) node{\tiny$-$};
\draw (2.5,4) node{\tiny$-$};
\draw (1.5,2.8) node{$a$};
\draw (2.2,1.5) node{$a$};
\draw (2.2,4) node{$b$};
\draw (4,2.8) node{$b$};
\draw (4,1) node{$(\lambda\leq0)$};
\draw[<->] (9,0) -- (9,5.5);
\draw[<->] (6.5,2.5) -- (12,2.5);
\draw[line width=1pt] (6.5,5.5) .. controls (7.7,1) and (7.9,1) .. (8,1.5);
\draw[line width=1pt] (10.5,4) .. controls (10.7,4.075) and (10.9,4.15) .. (12,0);
\draw[line width=1pt, dashed] (8,1.5) .. controls (8.25,2.5) and (8.75,2.5) .. (9,1.75) .. controls (9.25,1.25) and (9.9,1) .. (10,2.5) .. controls (10.2,3.75) and (10.45,4) .. (10.5,4);
\draw (1.5,1.5) node{$\bullet$};
\draw (4,4) node{$\bullet$};
\draw (8,1.5) node{$\bullet$};
\draw (10.5,4) node{$\bullet$};
\draw (10.3,1) node{${\tt z}={\tt P}_{\vec{u},\lambda}({\tt t})$};
\draw (8,2.5) node{\tiny$|$};
\draw (10.5,2.5) node{\tiny$|$};
\draw (9,1.5) node{\tiny$-$};
\draw (9,4) node{\tiny$-$};
\draw (8,2.8) node{$a$};
\draw (10.5,2.8) node{$b$};
\draw (8.7,1.5) node{$a$};
\draw (8.7,4) node{$b$};
\draw (10.3,0.5) node{$(\lambda>0)$};}
\end{tikzpicture}
\end{center}
\caption{Possible graphs of ${\tt P}_{\vec{u},\lambda}$ for $\lambda\in\R$}\label{fig:graphs}
\end{figure}

\noindent Consequently, ${\tt P}_{\vec{u},\lambda}(\R\setminus]a,b[)=\R\setminus]a,b[$. If $\lambda>0$ then $\lim_{t\to\pm\infty}(1-\lambda{\ttL}^2(t\vec{u}))=\mp\infty$ for all $t\in\R$. As $a<0<b$, we obtain ${\tt P}_{\vec{u},\lambda}(\R\setminus]a,b[)=\R$, as claimed.

\paragraph{}Summarizing the previous discussion, when $r=m-1$ we are able to represent the set $\R^n\setminus(\Int(p)\times]-\infty,0])$ as a polynomial image of $\R^n\setminus\Int(\pol_{m,\times})$. This will allow us to apply the polynomial map ${\tt F}:\R^n\to\R^n$ that we are about to construct in \ref{ss:const} to obtain
$$
{\tt F}(\R^n\setminus(\Int(\p)\times{]{-\infty},0]}))=\R^n\setminus(\Int(\p)\times{]{-\infty},0[})=\R^n\setminus\Int(\pol).
$$
Observe that $\Int(\p)\times{]{-\infty},0]}=\Int(\pol)\cup\Lambda$, where $\Lambda$ is the interior of the facet $\Ff_m:=H_m\cap\pol$.

\subsubsection{Construction of the polynomial map ${\tt F}$.}\label{ss:const} Set $\x':=(\x_1,\dots,x_{n-1})$ and $\x:=(\x_1,\dots,\x_n)=(\x',\x_n)$. Let ${\tth}_i(\x):={\ttg}_i(\x')-\epsilon_i\x_n$ be a linear equation of $H_i$ for $i=1,\dots,m-1$ where $\epsilon_i=0$ for $1\le i\le r$ (hyperplanes containing vertical facets of $\pol$) and $\epsilon_i=1$ for $r+1\le i\le m-1$ (hyperplanes containing non-vertical facets). Observe that $H_i^+=\{{\tth}_i\geq0\}$ and ${\ttg}_i({\bf0}')>0$ for $i=r+1,\ldots,m-1$ because the origin of $\R^n$ belongs to the interior of $\Ff_m$. For $j=1,\ldots,r$ denote a linear equation of $H_j$ with ${\tth}_j(\x):={\ttg}_j(\x')$ and assume that $H_j^+=\{{\ttg}_j\geq0\}$ for each $j$. We introduce now some auxiliary polynomials that will appear in the polynomial maps we are about to construct. Consider first
\begin{equation}\label{eq:q}
{\tt Q}(\x):=\x_n-\|\x'\|^2-1-
\sum_{i=r+1}^{m-1}\Big(\frac{{\ttg}_i^2(\x')+1}{2}\Big),
\end{equation}
where the value of the last summation term becomes $0$ when $r=m-1$. The polynomial ${\tt Q}(\x)$ has two properties of interest to us. First, \em the region ${\tt Q}(\x)\ge 0$ lies `above' all the hyperplanes containing non-vertical facets of $\pol$ and `above' the hypersurface $x_n=\|x'\|^2+1$\em. In addition, this region is connected and projects onto $\{x_n=0\}$. Second, \em ${\tt Q}(\x)$ is always negative on $\{x_n\le 0\}$\em. Next, we introduce
\begin{equation}\label{eq:g}
{\tt G}(\x):=\Bigr(\prod_{j=1}^r{\tth}_j(\x')
\Bigr)^2\cdot\Bigr(
\prod_{i=r+1}^{m-1}{\tth}_i(\x)\Bigr),
\end{equation}
where the second factor becomes $1$ when $r=m-1$. \em This polynomial function vanishes on the hyperplanes containing the facets of $\pol$ and is positive on its interior\em. In addition, \em vertical facets of $\pol$ does not change the sign of ${\tt G}(\x)$\em. Finally, consider the polynomial
\begin{equation}\label{eq:p}
{\tt P}(\x):=1-{\tt Q}(\x){\tt G}^2(\x),
\end{equation}
which will play a main role in the second part \ref{proof:2dim} of this proof.

\subsubsection{} Define with the aid of \eqref{eq:p}
\begin{equation}\label{eq:F}
{\tt F}_1(\x):=\x'(({\tt P}(\x)-1)^2+{\tt P}^2(\x))\in\R[\x]^{n-1}\quad\text{and}\quad{\tt F}_2(\x)=\x_n{\tt P}^2(\x)\in\R[\x].
\end{equation}
We want to prove the following: \em The map ${\tt F}:\R^n\to\R^n,\ x\mapsto ({\tt F}_1(x),{\tt F}_2(x))$ satisfies \em
$$
\begin{cases}
{\tt F}(\R^n\setminus\Int(\pol_{m,\times}))=\R^n\setminus\Int(\pol)&
\quad\text{if $r<m-1$},\\
{\tt F}(\R^n\setminus(\Int(\pol)\cup\Lambda))
=\R^n\setminus\Int(\pol)&\quad\text{if $r=m-1$}.
\end{cases}
$$
Recall here that $\pol_{m,\times}:=\bigcap_{i=1}^{m-1}H_i^+$.

\subsubsection{Reduction to the $2$-dimensional case} Consider the family of vertical hyperplanes through the origin. This family can be parametrized as follows: \em for each $\vec{u}\in\sph^{n-2}\subset\R^{n-1}$ define $\pi_{\vec{u}}:=\{(\vec{u}y,z):\ (y,z)\in\R^2\}$, which is the plane through the origin generated by the vectors $\vec{u}$ and $\vec{e}_n$\em. Obviously $\R^n=\bigcup_{\vec{u}\in\sph^{n-2}}\pi_{\vec{u}}$ and ${\tt F}(\pi_{\vec{u}})\subset\pi_{\vec{u}}$. Therefore it is enough to check that
$$
\begin{cases}
{\tt F}(\pi_{\vec{u}}\setminus\Int(\pol_{m,\times}))=\pi_{\vec{u}}\setminus\Int(\pol)&
\quad\text{if $r<m-1$},\\
{\tt F}(\pi_{\vec{u}}\setminus(\Int(\pol)\cup\Lambda))
=\pi_{\vec{u}}\setminus\Int(\pol)&\quad\text{if $r=m-1$}.
\end{cases}
$$
for all $\vec{u}\in\sph^{n-2}$. As $\Int(\pol_{m,\times})$ and $\Int(\pol)$ are open subsets of $\R^n$, we have 
$$
\Int(\pol_{m,\times})\cap\pi_{\vec{u}}=\Int(\pol_{m,\times}\cap\pi_{\vec{u}})\quad\text{and}\quad\Int(\pol)\cap\pi_{\vec{u}}=\Int(\pol\cap\pi_{\vec{u}}).
$$
As the origin of $\R^n$ belongs to the interior of ${\Ff}_m$, the intersection $\Ee_{m,\vec{u}}:=\Ff_m\cap\pi_{\vec{u}}$ is an edge of $\pol\cap\pi_{\vec{u}}$ and its interior $\Delta_{\vec{u}}=\Lambda\cap\pi_{\vec{u}}$ contains the origin of $\R^n$. We are reduced to prove that
\begin{equation}\label{eq:red}
\begin{cases}
{\tt F}(\pi_{\vec{u}}\setminus\Int(\pol_{m,\times}\cap\pi_{\vec{u}}))=\pi_{\vec{u}}\setminus\Int(\pol\cap\pi_{\vec{u}})&
\quad\text{if $r<m-1$},\\
{\tt F}(\pi_{\vec{u}}\setminus(\Int(\pol\cap\pi_{\vec{u}})\cup\Delta_{\vec{u}}))
=\pi_{\vec{u}}\setminus\Int(\pol\cap\pi_{\vec{u}})&\quad\text{if $r=m-1$},
\end{cases}
\end{equation}
for all $\vec{u}\in\sph^{n-2}$. Fix now an arbitrary $\vec{u}\in\sph^{n-2}$ and write
\begin{equation}\label{eq:bunch}
\begin{aligned}
&{\ttl}_i(\y,\z):={\tth}_i(\y\vec{u},\z)={\ttg}_i(\y\vec{u})-\epsilon_i\z
={\tts}_i(\y)-\epsilon_i\z\quad\forall\,i=1,\ldots,m,\\
&{\ttq}(\y,\z):={\tt Q}(\y\vec{u},z)=\z-\y^2-1-\sum_{i=r+1}^{m-1}\frac{{\tts}_i^2(\y)+1}{2},\\
&{\tt g}(\y,\z):={\tt G}(\y\vec{u},z)=\Bigr(\prod_{j=1}^r{\tts}_j
(\y)\Bigr)^2\cdot\Bigr(
\prod_{i=r+1}^{m-1}{\ttl}_i(\y,\z)\Bigr),\\
&{\tt p}(\y,\z):={\tt P}(\y\vec{u},\z)=1-{\ttq}(\y,\z){\ttgg}^2(\y,\z),\\
&{\ttf}_1(\y,\z):={\tt F}_1(\y\vec{u},\z)=\y(({\tt p}(\y,\z)-1)^2+{\tt p}^2(\y,\z)),\\
&{\ttf}_2(\y,\z):={\tt F}_2(\y\vec{u},\z)=\z{\tt p}^2(\y,\z).
\end{aligned}
\end{equation}
The linear polynomials $\ttl_i$ can be interpreted as the restrictions to the plane $\pi_{\vec{u}}$ of the linear polynomials $\tth_i$. We have settled in $\pi_{\vec{u}}\equiv\R^2$ coordinates with respect to the vectors $\{\vec{u},\vec{e}_n\}$. Analogously, the functions $\ttq$, $\tt g$, $\tt p$, $\ttf_1$ and $\ttf_2$ can be understood respectively as the restrictions to the plane $\pi_{\vec{u}}$ of the polynomials $\tt Q$, $\tt G$, $\tt P$, ${\tt F}_1$ and ${\tt F}_2$, appearing in \eqref{eq:q} through \eqref{eq:F}.

As ${\tth}_i({\bf{0}})={\ttg}_i({\bf{0}}')>0$, the linear equation ${\tts}_i$ is not identically zero for $i=r+1,\ldots,m-1$. Denote 
\begin{equation}\label{eq:pol}
\p:=\bigcap_{i=1}^m\{{\ttl}_i\geq0\}\subset\R^2,\quad\p_{m,\times}=\bigcap_{i=1}^{m-1}\{{\ttl}_i\geq0\}\subset\R^2, 
\end{equation}
$\Ee_m:=\p\cap\{\z=0\}$ and $\Delta:=\Int(\Ee_m)$. Recall that the origin of $\R^n$ belongs to the interior of $\Ff_m$, the origin of $\R^2$ belongs to $\Delta$ and $\Ee_m$ is an edge of $\p$. Notice that the geometrical objects $\p$, $\p_{m\times}$, $\Ee_m$ and $\Delta$ can be interpreted respectively as the intersections of $\pol$, $\pol_{m,\times}$, $\Ff_m$ and $\Int(\Ff_m)$ with the plane $\pi_{\vec{u}}$. 

\subsubsection{}\label{reductionf}Restating \eqref{eq:red} for the plane $\R^2$, we are left to prove the following result.

\begin{thm}\label{main1dim2}
Consider the polynomial map ${\ttf}:=({\ttf}_1,{\ttf}_2):\R^2\to\R^2$ and let $\Delta:=\Int(\Ee_m)$. Then 
$$
\R^2\setminus\Int(\p)=\begin{cases}
{\ttf}(\R^2\setminus\Int(\p_{m,\times}))&\text{if $r<m-1$,}\\
{\ttf}(\R^2\setminus(\Int(\p)\cap\Delta))&\text{if $r=m-1$.}
\end{cases}
$$
\end{thm}
Once this is shown the proof of Theorem \ref{main1} will be concluded.\qed

\subsection{Proof of Theorem \ref{main1dim2}}\label{proof:2dim}

For the sake of clearness we continue to denote in this section the coordinates of $\R^2$ with $(\y,\z)$. We keep the notation and equations introduced in \eqref{eq:bunch} and \eqref{eq:pol}. Note that the polynomials $\ttl_k$ can be written as follows: 
$$
{\ttl}_k(\y,\z)={\tts}_k(\y)-\sigma_i\z:=\begin{cases}
{\tts}_k(\y):=c_k(\y-b_k)&\text{if $k=1,\ldots,r$,}\\
{\tts}_k(\y)-\z:=a_k\y+b_k-\z&\text{if $k=r+1,\ldots,m-1$,}\\
-\z&\text{if $k=m$},
\end{cases}
$$
where $\tts_k$ is not identically zero for $k=1,\ldots,m-1$. Recall the following assumptions: \em The line $\{{\ttl}_m=0\}$ contains the edge $\Ee_m$ of $\p$, the origin of $\R^2$ is contained in the interior of $\Ee_m$ and $-\vec{e}_2\equiv(0,-1)\in\conv{\p}{}$\em. The reader should have in mind Figure~\ref{im:poly}, which sketches the behavior of the polynomial map $\ttf$ and helps to understand how it acts on $\R^2$.

\begin{figure}[ht]
\begin{minipage}[b]{0.32\linewidth}
\centering
\begin{tikzpicture}[scale=0.37,baseline=(current bounding box.north)]
\draw[top color=gray!50,bottom color=gray!20,fill opacity=1](0.5,3.5) to (0.5,16) to (14,16) to (14,3.5) to (10,3.5) to (5,8.5) to (5,3.5) to (0.5,3.5);
\draw[line width=0.5pt,dashed,->] (5,16) to (5,3.5);
\draw[line width=0.5pt,dashed,<-] (10,3.5) to (0.5,13);
\draw[line width=1] (0.5,13.5) parabola bend (7.5,10.5) (14,14.5);
\draw[line width=1.5pt,dashed] (0.5,14)..controls (1.8,12.9).. (2,12.8)..controls (3,12.2) and (3.6,12.7).. (4.5,15.7)to(4.6,16);
\draw[line width=1.5pt,dashed] (5.4,16)..controls (7.5,10)and(10,10.5)..(14,15.2);
\draw [line width=1.5pt,dashed] plot [smooth, tension=0.8] coordinates { (0.5,5) (5,4.5) (7.5,5.5) (7.2,3.8) (9,4.5) (14,5.5)};
\draw[fill=white,draw=none,opacity=0.7](10,3.5) to (5,8.5) to (5,3.5);
\draw[very thick,fill=gray!20,opacity=0.5](10,3.5) to (5,8.5) to (5,3.5);
\draw[<->] (0.5,6) -- (14,6);
\draw[<->] (6.5,3.5) -- (6.5,16);
\draw (9.2,12.6) node{{\tiny $X_\lambda,\ \lambda>0$}};
\draw (12,4.2) node{{\tiny $X_\lambda,\ \lambda<0$}};
\draw (9.5,9.8) node{{\tiny $\partial \Qq:=\{{\ttq}=0\}$}};
\draw (3.8,5.3) node{{$\p_{m,\times}$}};
\end{tikzpicture}
\vskip 0.1cm
\begin{tikzpicture}[scale=0.37,baseline=(current bounding box.north)]
\draw[white,->] (-1.5,0) -- (-1.5,2);
\draw[very thick,->] (0,2.5) -- (0,0);
\draw (0.8,1.25) node{$f$};
\end{tikzpicture}
\vskip 0.1cm
\begin{tikzpicture}[scale=0.37,baseline=(current bounding box.north)]
\draw[top color=gray!50,bottom color=gray!20,fill opacity=1](0.5,3.5) to (0.5,16) to (14,16) to (14,3.5) to (10,3.5) to (7.5,6) to (5,6) to (5,3.5) to (0.5,3.5);
\draw[line width=0.5pt,dashed,->] (5,16) to (5,3.5);
\draw[line width=0.5pt,dashed,<-] (10,3.5) to (0.5,13);
\draw[line width=1] (0.5,13.5) parabola bend (7.5,10.5) (14,14.5);
\draw[line width=1.5pt,dashed] (0.5,7.5) to (4.8,7.5);
\draw[line width=1.5pt,dashed] (5.2,7.5) to (14,7.5);
\draw[line width=1.5pt,dashed] (0.5,4.5) to (14,4.5);
\draw[fill=white,draw=none,opacity=0.7](10,3.5) to (7.5,6) to (5,6) to (5,3.5);
\draw[very thick,fill=gray!20,opacity=0.5](10,3.5) to (7.5,6) to (5,6) to (5,3.5);
\draw[<->] (0.5,6) -- (14,6);
\draw[<->] (6.5,3.5) -- (6.5,16);
\draw (11.5,8) node{{\tiny $f(X_\lambda),\ \lambda>0$}};
\draw (11.5,5) node{{\tiny $f(X_\lambda),\ \lambda<0$}};
\draw (9.5,9.8) node{{\tiny $f(\partial \Qq)=\partial\Qq$}};
\draw (4.4,5.3) node{$\p$};
\draw[fill=none,opacity=1] (5,7.5)circle(0.2);
\end{tikzpicture}
\vskip 0.35cm
\caption*{}
\end{minipage}
\hspace{0.5cm}
\begin{minipage}[b]{0.63\linewidth}
\centering
\begin{tikzpicture}[yscale=0.7,xscale=0.7,baseline=(current bounding box.north)]
\draw[thick,fill=gray!20,opacity=0.5,<->](5,0.5) to (5,8.5) to (13,0.5);
\draw[top color=gray!50,bottom color=gray!20,fill opacity=1]
(1.5,0.5) to (1.5,13)..controls (1.5,13.7) and (1.8,13.4).. (2,13.25)..controls (3,12.5) and (3.6,13).. (4.5,16) to (5.5,16)..controls (7.5,11.5)and(10,12)..(13.25,15.25)..controls(13.35,15.4)and(13.5,15.4)..(13.5,14.5)to(13.5,0.5)to(13,0.5)to(5,8.5)to(5,0.5)to(1.5,0.5);
\draw[<->] (0.5,6) -- (14.5,6);
\draw[<->] (6.5,0) -- (6.5,16);
\draw[line width=1.5pt] (0.5,13.5) parabola bend (7.5,10.5) (14.5,15);
 (5.4,16)..controls (7.5,10)and(10,10.5)..(14.5,15.8);
\draw[line width=1pt,dashed] (5,16) to (5,8.5);
\draw[line width=1pt,dashed] (5,8.5) to (0.5,13);
\draw[<-,bend left=-25, line width=1pt] (3,1.25) to (4.5,2);
\draw[<-,bend left=-25, line width=1pt] (3,2.25) to (4.5,3);
\draw[<-,bend left=-25, line width=1pt] (3,3.25) to (4.5,4);
\draw[<-,bend left=-25, line width=1pt] (3,4.25) to (4.5,5);
\draw[->,bend left=-20, line width=1pt] (12,2) to (13.4,1.5);
\draw[->,bend left=-20, line width=1pt] (11,3) to (12.4,2.5);
\draw[->,bend left=-20, line width=1pt] (10,4) to (11.4,3.5);
\draw[->,bend left=-20, line width=1pt] (9,5) to (10.4,4.5);
\draw[top color=gray!80,bottom color=gray!50,fill opacity=0.50]
(2,13.25)..controls (3,12.5) and (3.6,13).. (4.5,16)to(4.8,16)to(4.8,6.3) ..controls(4.77,6)..(4.6,6)to(2.1,6)..controls(2,6)and(1.85,6)..(1.8,6.3)to(1.8,13.2)..controls(1.8,13.4)and(1.5,13.7)..(1.5,13)..controls(1.5,13.7)and(1.8,13.4)..(2,13.25);
\draw[top color=gray!80,bottom color=gray!50,fill opacity=0.25] (5.2,16) to (5.2,6.5) .. controls (5.21,6) and (5.23,6) .. (5.4,6) to (13,6) ..controls(13.1,6) and (13.2,6.1)..(13.2,6.4)to (13.2,15) to (13.25,15.25) ..controls(13.35,15.4)and(13.5,15.4).. (13.5,14.5)..controls(13.5,15.4)and(13.35,15.4)..(13.25,15.25)..controls(10,12)and(7.5,11.5)..(5.5,16) to (4.5,16);
\draw[top color=gray!40,bottom color=gray!60,fill opacity=0.85] (2,6) .. controls (2.1,6.0) and (2.28,6) .. (2.3,6.3) to (2.3,15.5) to (4.5,15.5) to (4.5,6.5) .. controls (4.51,6.2) and (4.52,6) .. (4.6,6);
\draw[top color=gray!40,bottom color=gray!60,fill opacity=0.85] (5.4,15.5) to (12.8,15.5) to (12.8,6.5) .. controls (12.8,6.1) and (12.9,6) .. (13,6) to (5.3,6) .. controls (5.4,6) and (5.4,6) .. (5.4,8) to (5.4,15.5);
\draw[color=white,very thick](4.4,16)to(5.6,16)to(5.6,16.1)to(4.4,16.1);
\draw (7.5,3.5) node{{\Large$\p$}} (10,8.5) node{{\large $f(\p_{m,\times})$}};
\end{tikzpicture}
\caption*{}
\end{minipage}
\vskip -1cm
\caption{A sketch of the behavior of the map ${\ttf}:\R^2\to\R^2$.}
\label{im:poly}
\end{figure}

\subsubsection{Some properties of the region $\Qq:=\{\ttq>0\}$.}\label{ss:regq}
It is interesting for us to take into account certain properties of $\Qq$:

(i) $\Qq\subset\{\z-\y^2-1>0\}\subset\{\z-|\y|>0\}\subset\{\z>0\}=\{{\ttl}_m<0\}$.

(ii) $\Qq\subset\bigcap_{i=r+1}^m\{{\ttl}_i<0\}$. In particular $\p\cap\Qq=\varnothing$ and $\p_{m,\times}\cap\Qq=\varnothing$ if $r<m-1$.

Assume $r<m-1$ and fix $i_0\in\{r+1,\ldots,m-1\}$ and $(y_0,z_0)\in\Qq$. We have ${\ttq}(y_0,z_0)>0$, so
$$
z_0>y_0^2+1+\sum_{i=r+1}^{m-1}\Big(\frac{\tts_i^2(y_0)+1}{2}\Big)\geq\frac{\tts_{i_0}^2(y_0)+1}{2}\geq \tts_{i_0}(y_0).
$$
Therefore, ${\ttl}_{i_0}(y_0,z_0)=\tts_{i_0}(y_0)-z_0<0$ and $(y_0,z_0)\in\{{\ttl}_{i_0}<0\}$.

(iii) Let $M>0$ be such that $1+|a_i|+|b_i|<M$ for $i=r+1,\ldots,m-1$. Then 
$$
\Qq\subset\bigcap_{i=r+1}^{m-1}\{M\z-|\tts_i(\y)-\z|>0\}.
$$
Fix $i\in\{r+1,\ldots,m-1\}$ and $(y_0,z_0)\in\Qq$. By (i) we have $|y_0|<z_0=|z_0|$ and $z_0>1$. Thus,
$$
|\tts_i(y_0)-z_0|=|a_iy_0+b_i-z_0|\leq|a_i||y_0|+|b_i|+|z_0|\leq|a_i||z_0|+|b_i||z_0|+|z_0|<Mz_0,
$$
as required.
\qed

\begin{lem}\label{lem0}
For each $y_0\in U:=\R\setminus\{b_1,\ldots,b_r\}$ there exists 
$$
z_1>z_0:=y_0^2+1+\sum_{i=r+1}^{m-1}
\frac{{\tts}_i^2(y_0)+1}{2}
$$ 
such that ${\tt p}(y_0,z_1)=0$. In particular, $(y_0,z_1)\in\Qq$.
\end{lem}
\begin{proof}
Consider the odd degree polynomial 
$$
{\tt p}_{y_0}(\z):={\tt p}(y_0,\z)=1-{\ttq}(y_0,\z){\ttgg}^2(y_0,\z).
$$ 
and observe that ${\ttq}(y_0,z_0)=0$. As $\lim_{z\to+\infty}{\tt p}_{y_0}(\z)={-\infty}$ and ${\tt p}_{y_0}(z_0)=1$, there exists $z_1>z_0$ such that ${\tt p}_{y_0}(z_1)=0$, as required.
\end{proof}

\subsubsection{Level curves of the polynomial ${\ttf}_2$.} We need to understand the level curves of ${\ttf}_2$. For each $\lambda\in\R$ consider the plane algebraic curve 
$$
X_\lambda:=\{{\ttf}_2=\lambda\}.
$$
The understanding of these algebraic curves will help us to prove later the equality ${\tt f}(X_\lambda\setminus\Int(\p_{m,\times}))=\{\z=\lambda\}\setminus\Int(\p)$, which essentially represents the core of the proof. The following two results provide semialgebraic parametrizations of (a part of) $X_\lambda$. The properties of these parametrizations depend strongly on the sign of $\lambda$.

\begin{lem}\label{lem1}
Assume $\lambda>0$ and define $\varphi_\lambda(y):=\max\{z\in\R:\,(y,z)\in X_\lambda\}$ for each $y\in U:=\R\setminus\{b_1,\ldots,b_r\}$. Then
\begin{itemize}
\item[(i)] $\varphi_\lambda:U\to\R$ is a Nash function and its graph $\Gamma_\lambda$ is contained in $X_\lambda\cap \Qq$.
\item[(ii)] $\lim_{y\to b_j}\varphi_\lambda(y)=+\infty$ for all $j=1,\ldots,r$.
\end{itemize}
\end{lem}

\begin{lem}\label{lem2}
Assume $\lambda<0$ and let $\pi:\R^2\to\R,\ (y,z)\mapsto y$. Then $X_\lambda\subset\R\times\,]\lambda-1,0[$ and there exists a continuous semialgebraic map $\psi_\lambda:=(\psi_{1,\lambda},\psi_{2,\lambda}):\R\to\R^2$ such that $\im(\psi_\lambda)\subset X_\lambda$ and $\lim_{t\to\pm\infty}\psi_{1,\lambda}(t)=\pm\infty$. In particular, $\pi(\im(\psi_\lambda))=\R$.
\end{lem}

\renewcommand{\theparagraph}{\thesubsubsection.\arabic{paragraph}}
\setcounter{secnumdepth}{5}

\subsubsection{Proof of Lemma \em \ref{lem1}}
Note first that: \em the function $\varphi_\lambda$ is well-defined on $U$\em. 

Fix $y_0\in U$ and observe that ${\ttq}(y_0,\z)$ is a polynomial of degree one. Thus, ${\tt p}(y_0,\z)$ is a polynomial of degree $\geq1$ and ${\ttf}_2(y_0,\z)$ is a polynomial of odd degree. Consequently, the set $\{{\ttf}_2(y_0,z)=b\}\subset\R$ is non-empty and finite, so the value $\varphi_\lambda(y_0)$ exists.
\vskip 2mm

\noindent(i) The proof is conducted in several steps:

\paragraph{}\label{derqg2}We claim: \em The partial derivative $\frac{\partial({\ttq}{\ttgg}^2)}{\partial\z}$ is strictly positive on $\Qq_0:=\{{\ttq}>0\}\cap(U\times\R)$\em. 

We have 
\begin{equation*}
\begin{split}
\frac{\partial({\ttq}{\ttgg}^2)}{\partial\z}&=\frac{\partial{\ttq}}{\partial\z}{\ttgg}^2+2{\ttq}{\ttgg}\frac{\partial{\ttgg}}{\partial\z}\\
&={\ttgg}^2+2{\ttq}\Bigr(\prod_{j=1}^r{\ttl}_j(\y,\z)\Bigr)^4
\Bigr(\prod_{i=r+1}^{m-1}{\ttl}_i(\y,\z)\Bigr)
\Bigr(-\sum_{k=r+1}^{m-1}\prod_{i\neq k}{\ttl}_i(\y,\z)\Bigr)\\
&={\ttgg}^2+2{\ttq}\Bigr(\prod_{j=1}^r{\ttl}_j(\y,\z)\Bigr)^4
\Bigr(\sum_{k=r+1}^{m-1}(\z-\tts_i(\y))\prod_{i\neq k}{\ttl}_i(\y,\z)^2\Bigr),
\end{split}
\end{equation*}
and by \ref{ss:regq}(ii) the previous equation is strictly positive on $\Qq_0$.

\paragraph{}\label{pert}Fix $(y_0,\varphi_\lambda(y_0))\in\Gamma_\lambda$ and denote 
$$
z_0:=y_0^2+1+\sum_{i=r+1}^{m-1}\Big(\frac{\tts_i^2(y_0)+1}{2}\Big).
$$
We claim: \em There exists $z_0<z_1<z_2$ such that ${\tt p}(y_0,z_1)=0$ and ${\ttf}_2(y_0,z_2)=\lambda$. Besides, the graph $\Gamma_\lambda$ is contained in $\Qq_0$ and ${\tt p}$ is strictly negative on $\Gamma_\lambda$\em. 

Write $\Qq\cap\{y=y_0\}=\{(y_0,z):\,z>z_0\}$. By Lemma \ref{lem0} there exists $z_1>z_0$ such that ${\tt p}(y_0,z_1)=0$, so ${\ttf}_2(y_0,z_1)=0$. As $\lambda>0$ and $\lim_{z\to+\infty}{\ttf}_2(y_0,z)=+\infty$ (because it is an odd degree polynomial with positive leading coefficient), there exists $z_2>z_1>z_0$ such that ${\ttf}_2(y_0,z_2)=\lambda$. Thus, 
\begin{equation}\label{z1}
\varphi_\lambda(y_0)\geq z_2>z_1>z_0, 
\end{equation}
so $(y_0,\varphi_\lambda(y_0))\in\Qq\cap\{y=y_0\}\subset\Qq_0$.

As $\frac{\partial({\ttq}{\ttgg}^2)}{\partial\z}>0$ on $\Qq_0$, the polynomial function ${\tt p}_{y_0}(z):={\tt p}(y_0,z)$ is strictly decreasing on the interval ${]z_0,+\infty)}$. As ${\tt p}_{y_0}(z_1)=0$ and $\varphi_\lambda(y_0)>z_1$, we deduce ${\tt p}_{y_0}(\varphi_\lambda(y_0))<0$.

\paragraph{}Finally we prove: \em The function $\varphi_\lambda$ is Nash on $U$\em. 

Pick $y_0\in U$. We know that $(y_0,\varphi_\lambda(y_0))\in\Qq_0$, so $\frac{\partial({\ttq}{\ttgg}^2)}{\partial\z}(y_0,\varphi_\lambda(y_0))>0$. In addition, we know that ${\tt p}(y_0,\varphi_\lambda(y_0))<0$. It follows that
$$
\frac{\partial {\ttf}_2}{\partial\z}(y_0,\varphi_\lambda(y_0))={\tt p}^2(y_0,\varphi_\lambda(y_0)+\Big(2\z{\tt p}\cdot\Big(-\frac{\partial({\ttq}{\ttgg}^2)}{\partial\z}\Big)\Big)(y_0,\varphi_\lambda(y_0))>0.
$$

By the Implicit Function Theorem \cite[2.9.8]{bcr} there exist 
\begin{itemize}
\item open bounded intervals $I,J\subset\R$ such that $y_0\in I$ and $\varphi_\lambda(y_0)\in J$, 
\item $I\times J\subset\Qq_0\cap\{{\tt p}<0\}$ and
\item a Nash function $\phi:I\to J$
\end{itemize}
such that $X_{\lambda}\cap(I\times J)=\{z=\phi(y)\}$. Let us check: \em After shrinking $I$, we have $\varphi_\lambda|_I=\phi|_I$\em.

Indeed, suppose by contradiction that there exists a sequence $\{y_k\}_{k\geq1}\subset I$ that converges to $y_0$ and $\varphi_{\lambda}(y_k)>\sup(J)$ for all $k\geq1$. As $I\times J\subset\Qq_0\cap\{{\tt p}<0\}$, we deduce 
$$
I\times{]\inf(J),+\infty[}\subset\Qq_0\cap\{{\tt p}<0\}
$$ 
because ${\tt p}$ is decreasing on the line $\Qq\cap\{\y=y\}$ for all $y\in U$ (see the end of the proof of \ref{pert}). By \ref{derqg2} 
$$
\frac{\partial{\tt p}^2}{\partial{\tt z}}(y,z)=-2{\tt p}(y,z)\frac{\partial({\ttq}{\ttgg}^2)}{\partial\z}(y,z)>0
$$
for all $(y,z)>I\times{]\inf(J),+\infty[}$. Thus, ${\tt p}^2(y_k,\varphi_{\lambda}(y_k))>{\tt p}^2(y_k,\sup(J))>0$ for all $k\geq1$. As ${\ttf}_2(y_k,\varphi_{\lambda}(y_k))=\lambda$,
$$
\sup(J)\leq\varphi_{\lambda}(y_k)=\frac{\lambda}{{\tt p}^2(y_k,\varphi_{\lambda}(y_k))}\leq\frac{\lambda}{{\tt p}^2(y_k,\sup(J))}
$$
As $\{y_k\}_{k\geq1}\cup\{y_0\}$ is compact, there exists $M>0$ such that $\sup(J)<\varphi_{\lambda}(y_k)<M$. As 
$$
K:=X_\lambda\cap(\cl(I)\times[\sup(J),M])
$$
is compact, we may assume that the sequence $\{(y_k,\varphi_{\lambda}(y_k))\}_{k\geq1}$ converges to $(y_0,t_0)\in K$, so 
$$
\varphi_\lambda(y_0)<\sup(J)\leq t_0\leq\varphi_\lambda(y_0),
$$
which is a contradiction. Thus, after shrinking $I$, it holds $(y,\varphi_\lambda(y))\in I\times J$ for all $y\in I$. Consequently, $\varphi_\lambda|_I=\phi|_I$.

Therefore, $\varphi_\lambda$ is Nash function because it is locally Nash and by its definition semialgebraic.

\noindent(ii) Fix $b_{j_0}$ and take $y_0$ close to $b_{j_0}$. Consider the algebraic curve 
$$
Y:=\{{\tt p}:=1-{\ttq}{\ttgg}^2=0\}.
$$
By \ref{pert} there exists $(y_0,z_1)\in Y\cap\Qq_0$ such that $z_1<\varphi_{\lambda}(y_0)$. By \ref{ss:regq} (iii) there exists $M>0$ such that
\begin{multline*}
1={\ttq}(y_0,z_1){\ttgg}^2(y_0,z_1)\le(z_1-{\ttq}_0(y_0))(Mz_1)^{2(m-1-r)}\prod_{j=1}^rc_j^4(y_0-b_j)^4\\
\leq z_1(Mz_1)^{2(m-1-r)}\prod_{j=1}^rc_j^4(y_0-b_j)^4.
\end{multline*}
Therefore,
$$
\frac{1}{\sqrt[2m-2r-1]{M^{2(m-1-r)}\prod_{j=1}^r(y_0-b_j)^4}}\leq z_1\leq\varphi_\lambda(y_0),
$$
and so $\lim_{y\to b_{j_0}}\varphi_\lambda(y)=+\infty$, as required.
\qed

\subsubsection{Proof of Lemma \em\ref{lem2}}
The proof is conducted in several steps:

\paragraph{}We show first: \em the algebraic curve $X_\lambda$ when $\lambda<0$ lies in the band $\{\lambda-1<z<0\}$\em. 

Pick $(y_0,z_0)\in X_\lambda$. As ${\ttq}<0$ on the half-plane $\{z<0\}$ and $\lambda<0$, we have
$$
0>z_0=\frac{\lambda}{(1-{\ttq}(y_0,z_0){\ttgg}^2(y_0,z_0))^2}\geq\lambda>\lambda-1.
$$

\paragraph{}Next we prove: $\pi(X_\lambda)=\R$. 

Fix $y_0\in\R$ and consider the univariate polynomial ${\tth}(\z):={\ttf}_2(y_0,\z)$. Then ${\tth}(0)=0>\lambda$ and using again the fact that ${\ttq}<0$ on $\{z<0\}$ we deduce that
\begin{equation}\label{ml}
{\tth}(\lambda-1)=(\lambda-1)(1-{\ttq}(y_0,\lambda-1){\ttgg}^2(y_0,\lambda-1))^2\le\lambda-1<\lambda. 
\end{equation}
By continuity there exists $z_0\in{]\lambda-1,0[}$ such that ${\ttf}_2(y_0,z_0)={\tth}(z_0)=\lambda$, so $(y_0,z_0)\in X_\lambda$. Hence, $y_0=\pi(y_0,z_0)\in\pi(X_\lambda)$ as claimed.

\paragraph{} By \cite[2.9.10]{bcr} the curve $X_\lambda$ is the disjoint union of a finite set ${\mathfrak F}$ and finitely many affine Nash manifolds $N_1,\ldots,N_p$, each Nash diffeomorphic to the open interval ${]0,1[}$. As $X_\lambda$ contains no vertical lines, we may assume that $\pi:N_i\to\R$ is a Nash diffeomorphism onto its image for $i=1,\ldots,p$.

\paragraph{} Let $M>0$ be such that $\pi({\mathfrak F})\subset{]-M,M[}$. Let $Y_1,\dots,Y_s$ be the connected components of $X_\lambda$. As $X_{\lambda}\subset\R\times{]\lambda-1,0[}$, the same happens for each $Y_\ell$. Of course each $Y_\ell$ is a closed subset of $\R^2$. We claim: \em Some $Y_\ell$ connects the vertical edges $E_1:=\{-M\}\times\,[\lambda-1,0]$ and $E_2:=\{M\}\times\,[\lambda-1,0]$ of the rectangle $\Rr:=[-M,M]\times\,[\lambda-1,0]$\em. 
 
To prove this claim we will make use of Janiszewski's Theorem (see \cite{ja} or \cite[Theorem A]{bing}): \em If $K_1$ and $K_2$ are compact subsets of the plane $\R^2$ whose intersection is connected, a pair of points that is separated by neither $K_1$ nor $K_2$ is neither separated by their union $K_1\cup K_2$.\em

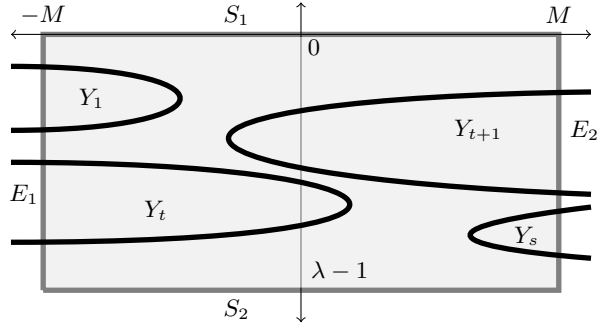
\begin{figure}[ht]
\begin{center}
{\Small
\begin{tikzpicture}[scale=0.848]

\draw[<->] (4.5,0) -- (4.5,5);
\draw[<->] (0,4.5) -- (9,4.5);

\draw[fill=gray!20, opacity=0.5,line width=2pt] (0.5,0.5) to (0.5,4.5) to (8.5,4.5) to (8.5,0.5) to (0.5,0.5);

\draw[line width=2pt] (0,1.25) .. controls (7.5,1.25) and (6.5,2.5) .. (0,2.5);
\draw[line width=2pt] (0,3) .. controls (3.5,3) and (3.5,4) .. (0,4);
\draw[line width=2pt] (9,3.6) .. controls (1.5,3.5) and (1.5,2.3) .. (9,2);
\draw[line width=2pt] (9,1.8) .. controls (6.5,1.5) and (6.5,1.2) .. (9,1);

\draw (0.2,2) node{$E_1$};
\draw (8.9,3) node{$E_2$};
\draw (3.5,4.8) node{$S_1$};
\draw (3.5,0.2) node{$S_2$};
\draw (4.7,4.3) node{$0$};
\draw (5.1,0.8) node{$\lambda-1$};
\draw (0.5,4.8) node{$-M$};
\draw (8.5,4.8) node{$M$};
\draw (2.25,1.75) node{$Y_t$};
\draw (1.25,3.5) node{$Y_1$};
\draw (7.25,3) node{$Y_{t+1}$};
\draw (8,1.35) node{$Y_s$};

\end{tikzpicture}}
\end{center}
\vspace*{-3mm}
\caption{Description of the situation}\label{fig:desc}
\end{figure}

Suppose by contradiction that no $Y_\ell$ connects $E_1$ with $E_2$. As $\pi(X_\lambda)=\R$, we may assume that the first $t<s$ connected components of $X_{\lambda}$ meet $E_1$, whereas the rest of them do not (Figure~\ref{fig:desc}). Define the compact sets
$$
K_1:=\bigcup_{\ell=1}^t(Y_\ell\cap\Rr)\cup E_1,\quad L:=\bigcup_{\ell=t+1}^s(Y_\ell\cap\Rr)
\quad \text{and}\quad K_2:=L\cup\partial\Rr,
$$
and note that the intersection $K_1\cap K_2=E_1$ is connected.

The horizontal segments $S_1:=[-M,M]\times\{0\}$ and $S_2:=[-M,M]\times\{\lambda-1\}$ satisfy ${\ttf}_2(S_1)=\{0\}$ and ${\ttf}_2(S_2)\subset\;]{-\infty},\lambda[$ (see \eqref{ml}), that is, $S_1\subset{\ttf}_2^{-1}(]\lambda,+\infty[)$ and $S_2\subset{\ttf}_2^{-1}(]{-\infty},\lambda[)$. Consequently, the algebraic curve $X_\lambda={\ttf}_2^{-1}(\lambda)$ separates the horizontal segments $S_1$ and $S_2$. Notice that if we restrict ${\ttf}_2$ to $\Rr$ the set $X_\lambda\cap\Rr$ still separates these segments on $\Rr$. Observe that $\partial \Rr=S_1\cup S_2\cup E_1\cup E_2$. As 
$$
K_1\cap\Big(\partial\Rr\cap\Big\{y\geq-\frac{M}{2}\Big\}\Big)=\varnothing\quad\text{and}\quad L\cap\Big(\partial\Rr\cap\Big\{y\leq\frac{M}{2}\Big\}\Big)=\varnothing,
$$
the real number
$$
\veps:=\min\Big\{\dist\Big(K_1,\Big(\partial\Rr\cap\Big\{y\geq-\frac{M}{2}\Big\}\Big)\Big),\dist\Big( L,\Big(\partial\Rr\cap\Big\{y\leq\frac{M}{2}\Big\}\Big)\Big)\Big\}
$$ 
is strictly positive and does not exceed $\frac{M}{2}$ (Figures~\ref{fig:pos1} and \ref{fig:pos2}). Let $0<\delta<\frac{\veps}{2}$ be such that $p_1:=(0,-\delta)\in{\ttf}_2^{-1}(]\lambda,+\infty[)$ and $p_2:=(0,\lambda-1+\delta)\subset{\ttf}_2^{-1}(]{-\infty},\lambda[)$. It holds that $K_1\cup K_2=(X_\lambda\cap\Rr)\cup\partial\Rr$ separates the points $p_1$ and $p_2$. Note that both $p_1,p_2$ belong to the open connected subsets
\begin{align*}
&W_1:=\Big\{p\in\Int(\Rr):\ 0<\dist\Big(p,\Big(\partial\Rr\cap\Big\{y\leq\frac{M}{2}\Big\}\Big)<\frac{\veps}{2}\Big\},\\
&W_2:=\Big\{p\in\Int(\Rr):\ 0<\dist\Big(p,\Big(\partial\Rr\cap\Big\{y\geq-\frac{M}{2}\Big\}\Big)<\frac{\veps}{2}\Big\}
\end{align*}
of $\Int(\Rr)$, while $K_1\cap W_2=\varnothing$ and $K_2\cap W_1=\varnothing$. Thus, $p_1,p_2$ are separated neither by $K_1$ nor by $K_2$, which contradicts Janiszewski's Theorem. The claim follows.

\noindent
\begin{figure}[ht]
\begin{minipage}[t]{0.49\textwidth}
{\Small\begin{tikzpicture}[scale=0.83]

\draw[fill=black!40!white,draw=none,fill opacity=1](2,4.5) arc (180:270:0.5cm) to (8,4) to (8,1) to (2.5,1) arc (90:180:0.5cm) to (8.5,0.5) to (8.5,4.5) to (2,4.5);
\draw[draw=none, fill=gray!20,opacity=0.5] (0.5,0.5) to (0.5, 4.5) to (8.5,4.5) to (8.5,0.5) to(0.5,0.5);
\draw[<->] (4.5,0) -- (4.5,5);
\draw[<->] (0,4.5) -- (9,4.5);

\draw[line width=2pt] (0.5,0.5) to (0.5,4.5);
\draw[line width=2pt,dashed] (0.5,4.5) to (8.5,4.5) to (8.5,0.5) to (0.5,0.5);

\draw[line width=2pt] (0,1.25) .. controls (7.5,1.25) and (6.5,2.5) .. (0,2.5);
\draw[line width=2pt] (0,3) .. controls (3.5,3) and (3.5,4) .. (0,4);


\draw (0.2,2) node{$E_1$};
\draw (8.9,3) node{$E_2$};
\draw (3.5,4.8) node{$S_1$};
\draw (3.5,0.2) node{$S_2$};
\draw (4.75,4.25) node{$0$};
\draw (5.1,0.75) node{$\lambda-1$};
\draw (0.5,4.8) node{$-M$};
\draw (8.5,4.8) node{$M$};
\draw (2.25,1.75) node{$Y_t$};
\draw (1.25,3.5) node{$Y_1$};
\draw (2.5,4.5) node{\tiny$|$};
\draw (6.5,4.5) node{\tiny$|$};
\draw (2.5,4.8) node{$-M/2$};
\draw (6.5,4.8) node{$M/2$};
\draw (4.5,0.75) node{\tiny$\bullet$};
\draw (4.5,4.25) node{\tiny$\bullet$};
\draw (4.2,0.75) node{$p_1$};
\draw (4.2,4.25) node{$p_2$};
\draw (8,2.5) node{$W_2$};
\draw (3,2.8) node{$K_1$};

\end{tikzpicture}
\vspace*{-7mm}
\captionof{figure}{Positions of $K_1$ and $W_2$}\label{fig:pos1}
}

\end{minipage}
\hfill
\begin{minipage}[t]{0.49 \textwidth}
{\Small
\begin{tikzpicture}[scale=0.83]

\draw[fill=black!40!white,draw=none,fill opacity=0.75](7,4.5) arc (0:-90:0.5cm) to (1,4) to (1,1) to (6.5,1) arc (90:0:0.5cm) to (0.5,0.5) to (0.5,4.5) to (6.5,4.5);
\draw[draw=none, fill=gray!20,opacity=0.5] (0.5,0.5) to (0.5, 4.5) to (8.5,4.5) to (8.5,0.5) to(0.5,0.5);
\draw[<->] (4.5,0) -- (4.5,5);
\draw[<->] (0,4.5) -- (9,4.5);

\draw[line width=2pt] (0.5,0.5) to (0.5,4.5) to (8.5,4.5) to (8.5,0.5) to (0.5,0.5);

\draw[line width=2pt] (9,3.6) .. controls (1.5,3.5) and (1.5,2.3) .. (9,2);
\draw[line width=2pt] (9,1.8) .. controls (6.5,1.5) and (6.5,1.2) .. (9,1);

\draw (0.2,2) node{$E_1$};
\draw (8.9,3) node{$E_2$};
\draw (3.5,4.8) node{$S_1$};
\draw (3.5,0.2) node{$S_2$};
\draw (4.75,4.25) node{$0$};
\draw (5.1,0.75) node{$\lambda-1$};
\draw (0.5,4.8) node{$-M$};
\draw (8.5,4.8) node{$M$};
\draw (7.25,3) node{$Y_{t+1}$};
\draw (8,1.35) node{$Y_s$};
\draw (2.5,4.5) node{\tiny$|$};
\draw (6.5,4.5) node{\tiny$|$};
\draw (2.5,4.8) node{$-M/2$};
\draw (6.5,4.8) node{$M/2$};
\draw (4.5,0.75) node{\tiny$\bullet$};
\draw (4.5,4.25) node{\tiny$\bullet$};
\draw (4.2,0.75) node{$p_1$};
\draw (4.2,4.25) node{$p_2$};
\draw (1,2.5) node{$W_1$};
\draw (6.5,1.75) node{$K_2$};
\end{tikzpicture}

\vspace*{-3mm}
\captionof{figure}{Positions of $K_2$ and $W_1$}\label{fig:pos2}
}

\end{minipage}
\end{figure}

\paragraph{} Let $\alpha:[-M,M]\to X_\lambda$ be a continuous semialgebraic path such that $\alpha(-M)\in X_\lambda\cap E_1$ and $\alpha(M)\in X_\lambda\cap E_2$. As $\pi({\mathfrak F})\subset{]-M,M[}$, we may assume that 
\begin{itemize}
\item $\alpha(-M)\in N_1$ and $\alpha(M)\in N_2$,
\item $\pi(N_1)=({-\infty},-M[$ and $\pi(N_2)=]M,\infty)$, where $\pi|_{N_1}$ and $\pi|_{N_2}$ are homeomorphisms.
\end{itemize}

\begin{figure}[ht]
{\Small\begin{tikzpicture}[scale=0.83]
\draw[draw=none, fill=gray!20,opacity=0.5] (1.5,0.5) to (1.5, 4.5) to (9.5,4.5) to (9.5,0.5) to(1.5,0.5);
\draw[<->] (5.5,0) -- (5.5,5);
\draw[<->] (0,4.5) -- (11,4.5);

\draw[line width=2pt] (1.5,0.5) to (1.5,4.5) to (9.5,4.5) to (9.5,0.5) to (1.5,0.5);

\draw[line width=2pt] (0,4) .. controls (7,4) and (7,3) .. (4,2.5) .. controls (2,2.2) and (2,1.5) .. (4,1.5).. controls (6,1.5) and (7,2) .. (7,2) .. controls (8,2.45) and (9.5,2.6) .. (11,2.5);
\draw[line width=2pt] (0,0.75) .. controls (4.5,1) and (4.5,1.25) .. (0,1.75);
\draw[line width=2pt] (11,4) .. controls (6,4) and (6,3) .. (11,3);

\draw (5.7,4.2) node{$0$};
\draw (6.1,0.8) node{$\lambda-1$};
\draw (1.5,4.8) node{$-M$};
\draw (9.5,4.8) node{$M$};
\draw (1.5,3.975) node{$\bullet$};
\draw (9.5,2.5) node{$\bullet$};
\draw (6.2,2.75) node{$X_\lambda$};
\draw (4,2) node{$\alpha(t)$};
\draw (0,3.65) node{$(\pi|_{N_1})^{-1}(t)$};
\draw (11,2.1) node{$(\pi|_{N_2})^{-1}(t)$};
\end{tikzpicture}}
\caption{Construction of the parametrization $\psi_\lambda$}\label{fig:par}
\end{figure}
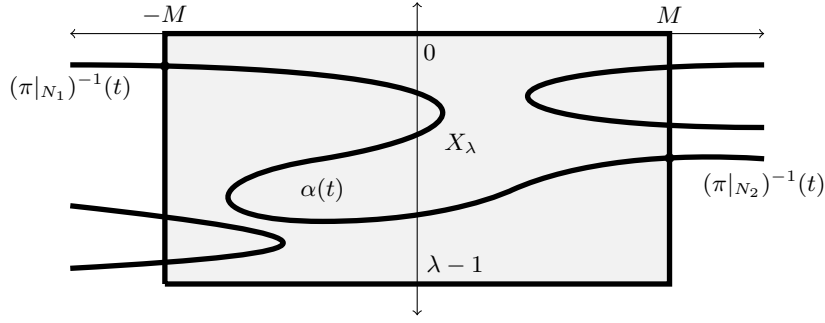

Finally, the continuous semialgebraic map (Figure~\ref{fig:par})
$$
\psi_\lambda(t)=(\psi_{1,\lambda}(t), \psi_{2,\lambda}(t)):=\begin{cases}
(\pi|_{N_1})^{-1}(t)&\text{it $t<-M$,}\\
\alpha(t)&\text{if $-M\le t\le M$,}\\
(\pi|_{N_2})^{-1}(t)&\text{it $t>M$,}\\
\end{cases}
$$
satisfies $\pi(\text{Im}(\psi_\lambda))=\R$ and $\lim_{t\to\pm\infty}(\psi_{1,\lambda}(t))=\pm\infty$, as required.
\qed

\subsubsection{Proof of Theorem \em\ref{main1dim2}}
Denote $\ol{\Ss}:=\R^2\setminus\Int(\p)$ and
$$
\ol{\Tt}:=\begin{cases}
\R^2\setminus\Int(\p_{m,\times})&\text{if $r<m-1$,}\\
\R^2\setminus(\Int(\p)\cup\Delta)&\text{if $r=m-1$.}
\end{cases}
$$ 
We have to prove that ${\ttf}(\ol{\Tt})=\ol{\Ss}$. Write $\overline{\Tt}=\overline{\Tt}_1\cup\overline{\Tt}_2\cup\overline{\Tt}_3$ and $\overline{\Ss}=\overline{\Ss}_1\cup\overline{\Ss}_2\cup\overline{\Ss}_3$ where
\begin{eqnarray*}
\overline{\Tt}_1:=\overline{\Tt}\cap\{z>0\},&
\quad & \overline{\Ss}_1:=\overline{\Ss}\cap\{z>0\},\\
\overline{\Tt}_2:=\overline{\Tt}\cap\{z=0\},&
\quad & \overline{\Ss}_2:=\overline{\Ss}\cap\{z=0\},\\
\overline{\Tt}_3:=\overline{\Tt}\cap\{z<0\},&
\quad & \overline{\Ss}_3:=\overline{\Ss}\cap\{z<0\}.
\end{eqnarray*}
Note that 
$$
\ol{\Ss}_1\sqcup\ol{\Ss}_2=\ol{\Ss}
\cap\{\z\geq0\}=\{\z\geq0\}\quad\text{and}
\quad\ol{\Tt}_3=\ol{\Ss}_3=
\{\z<0\}\setminus\Int(\p).
$$
It is enough to show:
\begin{eqnarray*}
{\ttf}(\overline{\Tt}_1\sqcup\overline{\Tt}_2)=
\overline{\Ss}_1\sqcup\overline{\Ss}_2,&\quad& {\ttf}(\overline{\Tt}_3)=\overline{\Ss}_3.
\end{eqnarray*}
The inclusion
$
{\ttf}(\ol{\Tt}_1\cup\ol{\Tt}_2)\subset
\ol{\Ss}_1\sqcup\ol{\Ss}_2$ is straightforward.
Therefore, we are left to show 
\begin{align}
\ol{\Ss}_1\sqcup\ol{\Ss}_2&\subset{\ttf}(\ol{\Tt}_1\sqcup\ol{\Tt}_2),\label{26}\\
\ol{\Ss}_3&\subset{\ttf}(\ol{\Tt}_3)\subset\ol{\Ss}_3.\label{27}
\end{align}
To prove \eqref{26} it is enough to check:
\begin{align*}
\{(b_1,0),\ldots,(b_r,0)\}&\subset{\ttf}(\ol{\Tt}_2),\\
\{\z=0\}\setminus\{(b_1,0),\ldots,(b_r,0)\}&\subset{\ttf}(\ol{\Tt}_1),\\
\text{and }\{\z>0\}&\subset{\ttf}(\ol{\Tt}_1).
\end{align*}

\paragraph{}\label{1}
We check first $\{(b_1,0),\ldots,(b_r,0)\}\subset{\ttf}(\ol{\Tt}_2)$. Note that $\{b_j\}\times\R\subset\ol{\Tt}$ because 
$$
\R^2\setminus\ol{\Tt}=\left\{
\begin{array}{ll}
\Int(\p_{m,\times})&\text{if $r<m-1$,}\\[4pt]
(\Int(\p)\cup\Delta)&\text{if $r=m-1$}
\end{array}
\right\}\subset\bigcap_{j=1}^r\{{\ttl}_j(\y,\z)=c_j(\y-b_j)>0\}. 
$$
As ${\ttf}(b_j,\lambda)=(b_j,\lambda)$, we have $\{b_j\}\times\R\subset{\ttf}(\ol{\Tt})$. Therefore $\{(b_1,0),\ldots,(b_r,0)\}\subset{\ttf}(\ol{\Tt}_2)$.

\paragraph{}Next we show: \em $U\times\{0\}=\{\z=0\}\setminus\{(b_1,0),\ldots,(b_r,0)\}\subset{\ttf}(\ol{\Tt}_1)$\em.

For each $y_0\in U:=\R\setminus\{b_1,\dots,b_r\}$ there exists by Lemma \ref{lem0} $z_1\in\R$ such that $(y_0,z_1)\in\Qq$ and ${\tt p}(y_0,z_1)=0$, so ${\ttf}_2(y_0,z_1)=0$ and ${\ttf}_1(y_0,z_1)=y_0$. By \ref{ss:regq} 
\begin{equation}\label{Q}
\Qq\subset\left\{
\begin{array}{ll}
\{\z>0\}\setminus\p_{m,\times}&\text{if $r<m-1$,}\\[4pt]
\{\z>0\}\setminus\p&\text{if $r=m-1$}
\end{array}
\right\}\subset\{\z>0\}\cap\ol{\Tt}=\ol{\Tt}_1. 
\end{equation}
Thus $U\times\{0\}\subset{\ttf}(\ol{\Tt}_1)$. 

\paragraph{}Let us prove: \em $\{\z>0\}\subset{\ttf}(\ol{\Tt}_1)$\em. To that end we show: \em If $\lambda>0$, the line $\{\z=\lambda\}$ is contained in ${\ttf}(\ol{\Tt}_1)$\em. 

Consider the curve $X_\lambda:=\{{\ttf}_2(y,z)=\lambda\}$. By Lemma \ref{lem1} there exists a Nash function $\varphi_\lambda:U:=\R\setminus\{b_1,\dots,b_r\}\to \R$ such that $\lim_{y\to b_j}\varphi_\lambda(y)=+\infty$ for $j=1,\dots,r$ and its graph $\Gamma_\lambda\subset X_\lambda\cap\Qq$. The latter condition means in particular that $\lim_{y\to\pm\infty}\varphi_\lambda(y)=+\infty$. 

Consider the function
$$
\Phi_\lambda:\R\to\R,\ y\mapsto
\begin{cases}
{\ttf}_1(y,\varphi_\lambda(y))&\text{if $y\in U$,}\\
b_j&\text{if $y=b_j$.}
\end{cases}
$$
Let us check: $\im(\Phi_\lambda)=\R$. It is enough to prove: \em $\Phi_\lambda$ is continuous and $\lim_{y\to\pm\infty}\Phi_\lambda(y)=\pm\infty$\em. 

Indeed, since $(y,\varphi_\lambda(y))\in X_\lambda$ we have
$$
\varphi_\lambda(y){\tt p}^2(y,\varphi_\lambda(y))=\lambda
\quad\leadsto\quad
{\tt p}(y,\varphi_\lambda(y))=\sqrt{\frac{\lambda}{\varphi_\lambda(y)}}.
$$
Thus, for each $y\in U$ we have by equation \eqref{sirve}
\begin{equation}\label{sqr}
\Phi_\lambda(y)={\ttf}_1(y,\varphi_\lambda(y))=y\Big(2\frac{\lambda}{\varphi_\lambda(y)}-2\sqrt{\frac{\lambda}{\varphi_\lambda(y)}}+1\Big),
\end{equation}
so $\lim_{y\to b_j}\Phi_\lambda(y)=b_j$ for $j=1,\ldots,r$ and $\Phi_\lambda$ is continuous. Also from equation \eqref{sqr} follows that $\lim_{y\to\pm\infty}\Phi_\lambda(y)=\pm\infty$. Thus, the map $\R\to\R\times\lambda=\{\z=\lambda\},\ y\mapsto(\Phi_{\lambda}(y),\lambda)$ is surjective. As the graph $\Gamma_\lambda$ of $\varphi_\lambda$ is contained in $\Qq\subset\{\z>0\}\cap\ol{\Tt}$ (see equation \eqref{Q}) and the points $(b_j,\lambda)\in\ol{\Tt}$ (see \ref{1}), we deduce $\{\z=\lambda\}\subset{\ttf}(\ol{\Tt}\cap\{\z>0\})={\ttf}(\ol{\Tt}_1)$.

Once it is proved \eqref{26} we proceed with \eqref{27}.

\paragraph{} Let us check: \em $\ol{\Ss}_3\subset{\ttf}(\ol{\Tt}_3)$\em. To that end we show: \em If $\lambda<0$, the difference $\{\z=\lambda\}\setminus\Int(\p)$ is contained in ${\ttf}(\ol{\Tt}_3)={\ttf}(\ol{\Tt}\cap\{\z<0\})$\em. 

Let $\psi_\lambda:\R\to X_\lambda$ be the continuous semialgebraic map constructed in Lemma \ref{lem2}. Write $\R=C\cup A$ where $C:=\psi_{\lambda}^{-1}(\im(\psi_{\lambda})\setminus\Int(\p))$ and $A:=\psi_{\lambda}^{-1}(\im(\psi_{\lambda})\cap\Int(\p))$. Notice that if $t_0\in\partial C$ there exist points $t_1\in C$ and $t_2\in A$ close to $t_0$. As $\psi_\lambda$ is continuous, $\psi_\lambda(t_1)\in\im(\psi_{\lambda})\setminus\Int(\p)$ and $\psi_\lambda(t_2)\in\Int(\p)$ are close to $\psi_\lambda(t_0)$, so $\psi_\lambda(t_0)\in X_\lambda\cap\partial\p$.

As $\psi_\lambda(t_0)\in\partial\p$ we have ${\ttgg}(\psi_\lambda(t_0))=0$, so ${\tt p}(\psi_\lambda(t_0))=1$ and ${\ttf}(\psi_\lambda(t_0))=\psi_\lambda(t_0)$. In addition, as $\psi_\lambda(t_0)\in X_\lambda$ we have ${\ttf}_2(\psi_\lambda(t_0))=\lambda$, so $\psi_\lambda(t_0)={\ttf}(\psi_\lambda(t_0))=(\psi_{1,\lambda}(t_0),\lambda)$. Thus,
$$
\psi_\lambda(t_0)\in\{\z=\lambda\}\cap\partial\p=\partial(\{\z=\lambda\}\setminus\Int(\p)).
$$
Notice that $\{\z=\lambda\}\setminus\Int(\p)=(S_1\times\{\lambda\})\sqcup(S_2\times\{\lambda\})$ where 
$$
S_1:=\begin{cases}
{]{-\infty},c_\lambda]}&\text{or}\\
\varnothing
\end{cases}
\quad\text{and}\quad
S_2:=\begin{cases}
{[d_\lambda,+\infty[}&\text{or}\\
\varnothing.
\end{cases}
$$

As $\pi(\im(\psi_\lambda))=\R$ and $\lim_{t\to\pm\infty}\psi_{1,\lambda}(t)=\pm\infty$, there exists two intervals $C_1$ and $C_2$ of $C$ (in case they are non-empty) such that
$$
C_1:=\begin{cases}
{]{-\infty},c'_\lambda]}&\text{if $S_1=]{-\infty},c_\lambda]$,}\\
\varnothing&\text{if $S_1=\varnothing$}
\end{cases}
\quad\text{and}\quad
C_2:=\begin{cases}
{[d'_\lambda,+\infty[}&\text{if $S_2=[d_\lambda,+\infty[,$}\\
\varnothing&\text{if $S_2=\varnothing$}
\end{cases}
$$
where $\psi_{1,\lambda}(c_\lambda')=c_\lambda$ and $\psi_{1,\lambda}(d_\lambda')=d_\lambda$ if the corresponding $S_i\neq\varnothing$. Let us show: \em ${\ttf}(\psi_\lambda(C_i))=S_i\times\{\lambda\}$ for $i=1,2$\em. As $\psi_\lambda(C)\subset\{\z<0\}\setminus\Int(\p)=\{\z<0\}\cap\ol{\Tt}=\ol{\Tt}_3$, we conclude 
$$
\{\z=\lambda\}\setminus\Int(\p)\subset{\ttf}(\psi_\lambda(C))\subset{\ttf}(\ol{\Tt}\cap\{\z<0\})={\ttf}(\ol{\Tt}_3).
$$

Indeed, rewrite ${\ttf}_1$ as follows
\begin{equation}\label{sirve}
{\ttf}_1=\y(({\tt p}-1)^2+{\tt p}^2)=\y(2{\tt p}^2-2{\tt p}+1)=\y\Big(2\Big({\tt p}^2-\frac{1}{2}\Big)^2+\frac{1}{2}\Big)=\y\Big(2\Big(\frac{1}{2}-{\ttq}{\ttgg}^2\Big)^2+\frac{1}{2}\Big).
\end{equation}
As ${\ttf}_2(\psi_\lambda(t))=\lambda$,
$$
{\tt p}^2(\psi_\lambda(t))=\frac{\lambda}{\psi_{2,\lambda}(t)}\quad\leadsto\quad{\ttf}_1(\psi_\lambda(t))=\psi_{1,\lambda}(t)\Big(2\Big(\frac{\lambda}{\psi_{2,\lambda}(t)}-\frac{1}{2}\Big)^2+\frac{1}{2}\Big).
$$
As $\lim_{t\to\pm\infty}\psi_{1,\lambda}(t)=\pm\infty$, we have $\lim_{t\to\pm\infty}{\ttf}_1(\psi_\lambda(t))=\pm\infty$. In addition,
$$
\left.
\begin{array}{l}
\psi_{\lambda}(c_\lambda')=(c_\lambda,\lambda)\\[4pt]
\psi_{\lambda}(d_\lambda')=(d_\lambda,\lambda)
\end{array}
\right\}\in\partial\p,
$$
so ${\ttf}_1(\psi_{\lambda}(c_\lambda'))=c_\lambda$ and ${\ttf}_1(\psi_{\lambda}(d_\lambda'))=d_\lambda$. Consequently, ${\ttf}_1(\psi_\lambda(C_i))=S_i$ for $i=1,2$, as required. 

\paragraph{} Finally we show: ${\ttf}(\ol{\Tt}_3)\subset\ol{\Ss}_3=\{\z<0\}\setminus\Int(\p)$. 

Notice first that ${\ttf}(\ol{\Ss}_3)\subset\{z<0\}$. Therefore, we only have to check that ${\ttf}(\ol{\Ss}_3)\cap\Int(\p)$ is the empty set. Let $(y_0,z_0)\in\ol{\Tt}_3=\ol{\Tt}\cap\{\z<0\}=\ol{\Ss}\cap\{\z<0\}=\{\z<0\}\setminus\Int(\p)$ and choose an edge of $\p$ and a linear equation ${\ttl}:=a\y+b\z+c$ of the line containing it that satisfies $\p\subset\{{\ttl}\geq0\}$ and ${\ttl}(y_0,z_0)\le 0$. As $-\vec{e}_2\equiv(0,-1)\in\conv{\p}{}$ and $(0,0)\in\p$, we have $b\leq0$ and $c\geq0$. Write $b=-\beta^2$ and $c=\gamma^2$. We have
\begin{equation}\label{ineq}
{\ttl}(y_0,z_0)=ay_0-\beta^2z_0+\gamma^2\leq0.
\end{equation}
We show next ${\ttl}({\ttf}(y_0,z_0))\leq0$. Indeed,
\begin{equation*}
\begin{split}
{\ttl}({\ttf}(y_0,z_0))=\,&ay_0((1-{\tt p})^2(y_0,z_0)+{\tt p}^2(y_0,z_0))-\beta^2z_0{\tt p}^2(y_0,z_0)+\gamma^2\\
=\,&(ay_0-\beta^2z_0+\gamma^2)((1-{\tt p})^2(y_0,z_0)+{\tt p}^2(y_0,z_0))\\
&+\beta^2z_0(1-{\tt p})^2(y_0,z_0)-\gamma^2((1-{\tt p})^2(y_0,z_0)+{\tt p}^2(y_0,z_0))\leq0
\end{split}
\end{equation*}
because $ay_0-\beta^2z_0+\gamma^2\leq0$ and $z_0<0$. Hence ${\ttf}(y_0,z_0)\not\in\Int(\p)$, as required.
\qed

\renewcommand{\theparagraph}{\thesubsection.\arabic{paragraph}}
\setcounter{secnumdepth}{4}

\section{Separation of distinguished semialgebraic sets}\label{s4}

A crucial step to prove Theorem \ref{main2} is roughly speaking the following separation result. Let $\pol\subset\R^n$ be a convex polyhedron of dimension $n$ and let $\p\subset\R^{n-1}$ be the projection of $\pol$ onto the first $(n-1)$ coordinates. Consider the `infinite prysm' $\Int(\p)\times\R$, which henceforth will be called \em skyscraper\em. Under mild conditions on the placement of $\pol$ in $\R^n$ the difference $(\Int(\p)\times\R)\setminus\pol$ has two connected components: the `attic' $\Cc^+$ of the skyscraper $\Int(\p)\times\R$ and its `basement' $\Cc^-$. It would be desirable (and much simpler for the exposition) to find a polynomial map ${\ttf}\in\R[\x_1,\ldots,\x_{n-1}]$ that separates the attic $\Cc^+$ and the basement $\Cc^-$ of the skyscraper $\Int(\p)\times\R$, but the strong restrictions concerning separation of non-compact semialgebraic sets by polynomial functions suggests to use more general functions. In Lemma \ref{prop:polP} we find a rational map depending on $(\x_1,\ldots,\x_{n-1})$ that separates $\Cc^+$ and $\Cc^-$, see Figure \ref{ratsep}. In addition its pole set is contained in its zero set and it does not intersect $\Int(\p)$.

\begin{figure}[ht]
\setlength{\unitlength}{10mm}
\begin{picture}(10,8)
\put(0,0){\includegraphics[width=10cm]{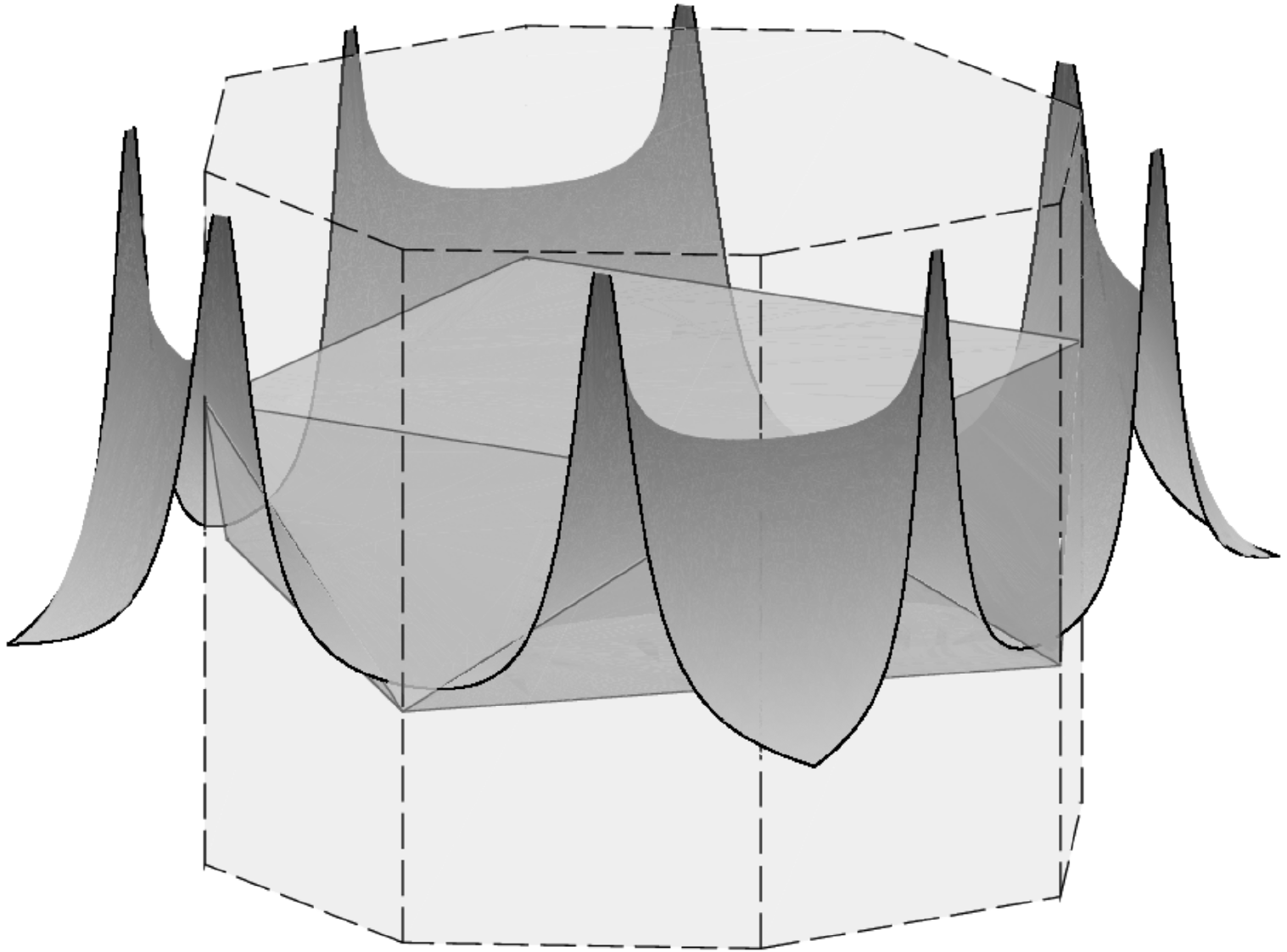}}
\put(3.5,3){$\pol$}
\put(3.5,4.8){$\Cc^+$}
\put(3.5,0.5){$\Cc^-$}
\put(5.9,6.5){$\Int(\p)\times\R$}
\put(0,4){\vector(1,-1){0.8}}
\put(-0.8,4.5){rational}
\put(-0.9,4.1){separator}
\end{picture}
\caption{Rational separator for the attic and the basement of a skyscraper}\label{ratsep}
\end{figure}

The result announced above is based on a preliminary one concerning rational separation of tuples of variables, which has interest by its own.

\subsection{Rational separation of tuples of variables}

Fix non-negative integers $r,s\geq1$ and set $y:=(y_1,\ldots,y_r)$ and $z:=(z_1,\ldots,z_s)$. Consider the convex polyhedron
$$
\Qq_{r,s}:=\{(y;z)\in\R^r\times\R^s:\,\max\{y_1,\ldots,y_r\}\leq\min\{z_1,\ldots,z_s\}\}.
$$
It holds
$$
\Int(\Qq_{r,s})=\{(y;z)\in\R^r\times\R^s:\,\max\{y_1,\ldots,y_r\}<\min\{z_1,\ldots,z_s\}\}.
$$
and $\cl(\Int(\Qq_{r,s}))=\Qq_{r,s}$. If $r,s\geq1$ and $k,\ell\geq0$, the map
\begin{equation*}
\begin{split}
\rho_{r,s}^{k,\ell}:\R^r\times\R^s&\to\R^{r+k}\times\R^{s+\ell},\\ 
(y;z)&\mapsto(y_1,\ldots,y_{r-1},y_r,\overset{(k+1)}{\ldots},y_r;z_1,\ldots,z_{s-1},z_s,\overset{(\ell+1)}{\ldots},z_s)
\end{split}
\end{equation*}
is a linear embedding such that $\rho(\Int(\Qq_{r,s}))\subset\Int(\Qq_{r+k,s+\ell})$ and $\rho(\Qq_{r,s})\subset\Qq_{r+k,s+\ell}$. Using that $\min(S)=-\max(-S)$ and $\max(S)=-\min(-S)$ for any finite set $S$, one proves that \em if $r,s\geq1$ the linear isomorphism
$$
\sigma:\R^r\times\R^s\to\R^s\times\R^r,\ (y;z)\mapsto(-z;-y)
$$
satisfies $\sigma(\Qq_{r,s})=\Qq_{s,r}$ and $\sigma(\Int(\Qq_{r,s}))=\Int(\Qq_{s,r})$\em.

Recall that a \em rational separator for the pair $(r,s)$ \em is a rational function $\phi_{r,s}:\R^r\times\R^s\dashrightarrow\R$ that is regular on $\Int(\Qq_{r,s})$, extends to a continuous (semialgebraic) function $\Qq_{r,s}$ and satisfies
$$
\max\{y_1,\ldots,y_r\}<\phi_{r,s}(y;z)<\min\{z_1,\ldots,z_s\}
$$
for each $(y;z)\in\Int(\Qq_{r,s})$. As $\cl(\Int(\Qq_{r,s}))=\Qq_{r,s}$ and $\phi_{r,s}$ extends to a continuous (semialgebraic) function $\Phi_{r,s}$ on $\Qq_{r,s}$, we deduce
$$
\max\{y_1,\ldots,y_r\}\leq\Phi_{r,s}(y;z)\leq\min\{z_1,\ldots,z_s\}
$$
for each $(y;z)\in\Qq_{r,s}$.

\subsubsection{Recursive properties of rational separators}
We present next some recursive properties of rational separators that will ease the proof of the existence of rational separators in Lemma \ref{lemma}.

\begin{remarks}\label{lemmar}
(i) Let $\phi_{r,s}:\R^r\times\R^s\dashrightarrow\R$ be a a rational separator for $(r,s)$. Then the rational function $\phi_{s,r}(z,y):=-\phi_{r,s}(-y;-z)$ is a rational separator for $(s,r)$. 

(ii) Let $r,s\geq1$ and $k,\ell\geq0$ and let $\phi_{r+k,s+\ell}:\R^{r+k}\times\R^{s+\ell}\dashrightarrow\R$ be a rational separator for $(r+k,s+\ell)$. Then $\phi_{r,s}(y,z)=\phi_{r+k,s+\ell}(\rho_{r,s}^{k,\ell}(y,z))$ is a rational separator for $(r,s)$.
\end{remarks}

\begin{lem}\label{indstep}
Let $r,s,k\geq1$ be such that $r\geq k$. Let $\phi_{r-k+1,s}$ be a rational separator for $(r-k+1,s)$ and $\phi_{k,s}$ a rational separator for $(k,s)$. Then 
$$
\phi_{r,s}(y;z)=\phi_{r-k+1,s}(y_1,\ldots,y_{r-k},\phi_{k,s}(y_{r-k+1},\ldots,y_r;z);z)
$$
is a rational separator for $(r,s)$.
\end{lem}
\begin{proof}
Define 
$$
\Mm:=\{(y_1,\ldots,y_{r-k},y_{r-k+1},\ldots,y_r;z)\in\R^r\times\R^s:\ \max\{y_{r-k+1},\ldots,y_r\}\leq\min\{z_1,\ldots,z_s\}\}.
$$
The rational map 
\begin{multline*}
\Theta:\R^r\times\R^s\dasharrow\R^{r-k+1}\times\R^s,\\ 
(y_1,\ldots,y_{r-k},y_{r-k+1},\ldots,y_r;z)\mapsto(y_1,\ldots,y_{r-k},\phi_{k,s}(y_{r-k+1},\ldots,y_r;z);z)
\end{multline*}
is regular on $\Int(\Mm)$ and extends continuously to $\Mm=\cl(\Int(\Mm))$. Clearly, $\Qq_{r,s}\subset\Mm$. We claim: $\Theta(\Int(\Qq_{r,s}))\subset\Int(\Qq_{r-k+1,s})$. 

If $(y;z)\in\Int(\Qq_{r,s})$, then $\max\{y_1,\ldots,y_r\}<\min\{z_1,\ldots,z_s\}$, so 
$$
\max\{y_{r-k+1},\ldots,y_r\}<\min\{z_1,\ldots,z_s\}
$$ 
and consequently $\phi_{k,s}(y_{r-k+1},\ldots,y_r;z)<\min\{z_1,\ldots,z_s\}$. Thus, 
\begin{equation}\label{mq}
\max\{y_1,\ldots,y_{r-k},\phi_{k,s}(y_{r-k+1},\ldots,y_r;z)\}<\min\{z_1,\ldots,z_s\}
\end{equation}
and $\Theta(y;z)\in\Int(\Qq_{r-k+1,s})$. Therefore $\phi_{r,s}$ is regular on $\Int(\Qq_{r,s})$ and extends continuously to $\Qq_{r,s}$. Let us check: \em $\max\{y_1,\ldots,y_r\}<\phi_{r,s}(y;z)<\min\{z_1,\ldots,z_s\}$ for each $(y;z)\in\Int(\Qq_{r,s})$\em. 

Indeed, by \eqref{mq} we know that $\Theta(y;z)\in\Int(\Qq_{r-k+1,s})$, so
\begin{multline*}
\max\{y_1,\ldots,y_r\}\leq
\max\{y_1,\ldots,y_{r-k},\phi_{k,s}(y_{r-k+1},\ldots,y_r;z)\}\\
<\phi_{r,s}(y;z):=\phi_{r-k+1,s}(y_1,\ldots,y_{r-k},\phi_{k,s}(y_{r-k+1},\ldots,y_r;z);z)
<\min\{z_1,\ldots,z_s\}.
\end{multline*}
We conclude that $\phi_{r,s}$ is a rational separator for $(r,s)$.
\end{proof}

\subsubsection{Existence of rational separators}
We prove next that for each pair of positive integers there exists a rational separator.

\begin{lem}\label{lemma}
For each pair $(r,s)$ of positive integers there exists a rational separator.
\end{lem}
\begin{proof}
The proof is conducted by induction on $t:=\min\{r,s\}$. We begin with some initial cases:

\noindent{\sc Case $t=r=s=2$.} Define
$$
\phi_{2,2}(y_1,y_2;z_1,z_2):=\frac{z_1z_2-y_1y_2}{(z_1+z_2)-(y_1+y_2)},
$$
which is regular in $\Int(\Qq_{2,2})$. Let us prove that if $(y_1,y_2;z_1,z_2)\in\Int(\Qq_{2,2})$, then
\begin{equation}\label{2222}
y_i<\phi_{2,2}(y_1,y_2;z_1,z_2)<z_j\quad\text{ for $i=1,2$ and $j=1,2$.}
\end{equation}
As $\phi_{2,2}$ is symmetric with respect to the variables $(y_1,y_2)$ and $(z_1,z_2)$, we only consider $i=1,j=1$. We have to prove
$$
y_1((z_1+z_2)-(y_1+y_2))<z_1z_2-y_1y_2<z_1((z_1+z_2)-(y_1+y_2)).
$$
The first inequality is equivalent to
$$
z_1z_2>y_1z_1+y_1z_2-y_1^2\ \iff\ (z_1-y_1)(z_2-y_1)>0,
$$
hence it holds. The second inequality is equivalent to
$$
-y_1y_2<z_1^2-y_1z_1-y_2z_1\ \iff\ z_1^2-y_1z_1-y_2z_1+y_1y_2=(z_1-y_1)(z_1-y_2)>0
$$
and it also holds. Note that $\Qq_{2,2}\cap\{z_1+z_2-y_1-y_2=0\}=\{z_1=z_2=y_1=y_2\}$. By \eqref{2222} the rational map $\phi_{2,2}$ extends continuously to $\Qq_{2,2}$ as
$$
\Phi_{2,2}(y_1,y_2;z_1,z_2)=\begin{cases}
\phi_{2,2}(y_1,y_2;z_1,z_2)& \text{if $z_1+z_2-y_1-y_2>0$,}\\
z_1 & \text{if $z_1=z_2=y_1=y_2$.}
\end{cases}
$$

\noindent{\sc Case $t=2$.} By Remark \ref{lemmar}(i) we may assume that $r\geq s=2$. We proceed by induction on $r$. We have constructed above a rational separator for $(2,2)$, so the initial case $r=2$ has been already approached. By induction hypothesis there exists a rational separator $\phi_{r-1,2}$ for $(r-1,2)$ if $r\geq3$. By Lemma \ref{indstep} the function
$$
\phi_{r,2}(y;z)=\phi_{r-1,2}(y_1,\ldots,y_{r-2},\phi_{2,2}(y_{r-1},y_r;z);z)
$$
is a rational separator for $(r,2)$.

\noindent{\sc Case $t\geq3$.} By Remark \ref{lemmar}(i) we may assume that $s\geq r\geq3$. Using Remark \ref{lemmar}(i) and the construction for $t=2$ we have a rational separator for $(2,s)$ if $s\geq2$. By induction hypothesis there exists a rational separator $\phi_{r-1,s}$ for $(r-1,s)$ if $r\geq3$. By Lemma \ref{indstep} the function
$$
\phi_{r,s}(y;s)=\phi_{r-1,s}(y_1,\ldots,y_{r-2},\phi_{2,s}(y_{r-1},y_r;z);z)
$$
is a rational separator for $(r,s)$.

\noindent{\sc Case $t=1$.} By Remark \ref{lemmar}(ii)
\begin{align*}
&\phi_{1,s}(y_1;z)=\phi_{2,s}(\rho_{1,s}^{1,0}(y_1;z)),\\
&\phi_{r,1}(y;z_1)=\phi_{r,2}(\rho_{r,1}^{0,1}(y;z_1))
\end{align*}
are respective rational separators for $(1,s)$ and $(r,1)$ if $r,s\geq1$.
\end{proof}

\subsection{Rational separation of the attic and the basement of a skyscraper}
Let $\pi_n:\R^n\to\R^{n-1},\ (x_1,\ldots,x_n)\mapsto(x_1,\ldots,x_{n-1})$ be the projection onto the first $(n-1)$ coordinates. The following position for an $n$-dimensional convex polyhedron $\pol\subset\R^n$ guarantees that the differences $(\Int(\p)\times\R)\setminus\pol$ and $(\Int(\p)\times\R)\setminus\Int(\pol)$, where $\p:=\pi_n(\pol)$, have each one two connected components.

\begin{define}
Let $\pol\subset\R^n$ be a convex polyhedron. We say that $\pol$ is in \em $\vec{\ell}_n$-bounded position \em if the intersection of $\pol$ with any vertical line $\ell$ is either empty or a bounded interval. 
\end{define}
\begin{remark}\label{cc}
One proves straightforwardly: \em $(\Int(\p)\times\R)\setminus\pol$ has two connected components\em.
\end{remark}

The following result provides an easy test to determine if a convex polyhedron is in $\vec{\ell}_n$-bounded position. As an straightforward consequence one shows that each $n$-dimensional convex polyhedron $\pol\subset\R^n$ with at least two facets can be placed in $\vec{\ell}_n$-bounded position.

\begin{lem}\label{test}
Let $\pol\subset\R^n$ be an $n$-dimensional convex polyhedron and let $\ell$ be a vertical line. Suppose that $\ell\cap\pol$ is a non-empty bounded segment, which may reduce to a point. Then $\pol$ is in $\vec{\ell}_n$-bounded position.
\end{lem}
\begin{proof}
Under the hypotheses $\vec{e}_n,-\vec{e}_n\not\in\conv{\pol}{}$, so the intersection of $\pol$ with any vertical line $\ell$ is either empty or a bounded interval.
\end{proof}

\begin{cor}\label{posfacet}
Let $\pol\subset\R^n$ be an $n$-dimensional convex polyhedron and let $\Ff$ be a facet of $\pol$ that has itself at least two facets. Then $\pol$ can be placed in $\vec{\ell}_n$-bounded position in such a way that $\Ff\subset\{\x_{n-1}=0\}$ and $\pol\subset\{\x_{n-1}\leq0\}$.
\end{cor}
\begin{proof}
After an affine change of coordinates we may assume that $\Ff\subset\{\x_{n-1}=0\}$ and $\pol\subset\{\x_{n-1}\leq0\}$. As $\Ff$ has itself at least two facets, there exist an affine change of coordinates that keeps invariant the half space $\x_{n-1}\leq0$ and such that $\Ff$ is in $\vec{\ell}_n$-bounded position (inside $\{\x_{n-1}=0\}$). By Lemma \ref{test} also $\pol$ is in $\vec{\ell}_n$-bounded position.
\end{proof}

\subsubsection{Convex polyhedron placed in $\vec{\ell}_n$-bounded position}\label{equations}
Let $\pol\subset\R^n$ be an $n$-dimensional convex polyhedron in $\vec{\ell}_n$-bounded position and let $\p:=\pi_n(\pol)\subset\R^{n-1}$ be its projection onto the hyperplane $\{x_n=0\}$. Let $\Ff_1,\ldots,\Ff_m$ be the facets of $\pol$ and let $H_i:=\{{\tth}_i=0\}$ be the hyperplane generated by $\Ff_i$. Assume $\pol=\bigcap_{i=1}^mH_i^+$ and set $\x:=(\x_1,\dots,\x_{n-1},\x_n):=(\x',\x_n)$. We can write
\begin{align}
&{\tth}_i(\x)=-{\tt a}_i(\x')+\x_n\quad\forall\,i=1,\ldots,r,\label{eq:h1}\\
&{\tth}_{r+j}(\x)={\tt b}_j(\x')-\x_n\quad\forall\,j=1,\ldots,s,\label{eq:h2}\\
&{\tth}_{r+s+k}(\x)={\tt c}_k(\x')\quad\forall\,k=1,\ldots,m-r-s.\label{eq:h3}
\end{align}
Equations \eqref{eq:h1} correspond to the non-vertical, lower facets of $\pol$ making up its `floor', while equations \eqref{eq:h2} and \eqref{eq:h3} correspond respectively to the non-vertical, upper facets making up its `ceiling' and those that constitute its vertical `walls'. As $\pol\subset\R^n$ is in $\vec{\ell}_n$-bounded position, we have $r,s\geq1$. Observe that $\pol=\pol_1\cap\pol_2$ where
\begin{align*}
&\pol_1:=\{(x',x_n)\in\R^n:\ \max\{{\tt a}_1(x'),\ldots,{\tt a}_r(x')\}\leq x_n\leq\min\{{\tt b}_1(x'),\ldots,{\tt b}_s(x')\}\},\\
&\pol_2:=\{(x',x_n)\in\R^n:\ {\tt c}_1(x')\geq0,\ldots,{\tt c}_{m-r-s}(x')\geq0\}.
\end{align*}
Notice that
\begin{align*}
&\pi_n(\pol_1)=\{x'\in\R^{n-1}:\ \max\{{\tt a}_1(x'),\ldots,{\tt a}_r(x')\}\leq\min\{{\tt b}_1(x'),\ldots,{\tt b}_s(x')\}\},\\
&\pi_n(\pol_2)=\{x'\in\R^{n-1}:\ {\tt c}_1(x')\geq0,\ldots,{\tt c}_{m-r-s}(x')\geq0\}.
\end{align*}

By \cite[II.Thm.6.5]{r} $\Int(\pol)=\Int(\pol_1)\cap\Int(\pol_2)$. As $\pi_n^{-1}(\pi_n(\pol_2))=\pol_2$, the equality $\pi_n(\pol)=\pi_n(\pol_1)\cap\pi_n(\pol_2)$ holds. Analogously, $\pi_n(\Int(\pol))=\pi_n(\Int(\pol_1))\cap\pi_n(\Int(\pol_2))$. In addition
\begin{align*}
&\Int(\pol_1)=\{(x',x_n)\in\R^n:\ \max\{{\tt a}_1(x'),\ldots,{\tt a}_r(x')\}<x_n<\min\{{\tt b}_1(x'),\ldots,{\tt b}_s(x')\}\},\\
&\Int(\pol_2)=\{(x',x_n)\in\R^n:\ {\tt c}_1(x')>0,\ldots,{\tt c}_{m-r-s}(x')>0\}.
\end{align*}
Notice that
\begin{align*}
&\pi_n(\Int(\pol_1))=\{x'\in\R^{n-1}:\ \max\{{\tt a}_1(x'),\ldots,{\tt a}_r(x')\}<\min\{{\tt b}_1(x'),\ldots,{\tt b}_s(x')\}\},\\
&\pi_n(\Int(\pol_2))=\{x'\in\R^{n-1}:\ {\tt c}_1(x')>0,\ldots,{\tt c}_{m-r-s}(x')>0\}.
\end{align*}

By \cite[II.Thm.6.6]{r} we have $\Int(\p)=\pi_n(\Int(\pol))$. Thus 
\begin{equation}\label{intk}
\Int(\p)=\pi_n(\Int(\pol))=\pi_n(\Int(\pol_1))\cap\pi_n(\Int(\pol_2)).
\end{equation}

By Remark \ref{cc} the difference $(\Int(\p)\times\R)\setminus\pol$ has two connected components. Namely
\begin{align*}
&\Cc^-:=\{(x',x_n)\in\Int(\p)\times\R:\ x_n<\max\{{\tt a}_1(x'),\ldots,{\tt a}_r(x')\}\}\quad\text{(basement)},\\
&\Cc^+:=\{(x',x_n)\in\Int(\p)\times\R:\ \min\{{\tt b}_1(x'),\ldots,{\tt b}_s(x')\}<x_n\}\quad\text{(attic)}.
\end{align*}

The next result provides a rational function $\frac{{\ttf}_2}{{\ttf}_1}\in\R(\x_1,\ldots,\x_{n-1})$ that separates $\Cc^-,\Cc^+$ and that is regular on $\Int(\p)$.

\begin{lem}[Rational separation of the attic and the basement of a skyscraper]\label{prop:polP}
There exists a polynomial ${\tt P}(\x):={\ttf}_1(\x')\x_n-{\ttf}_2(\x')\in\R[\x',\x_n]$ such that 
\begin{itemize}
\item $\{{\ttf}_1=0\}\subset\{{\ttf}_2=0\}$, 
\item ${\tt P}|_{\Cc^-}<0$ and ${\tt P}|_{\Cc^+}>0$,
\item ${\ttf}_1|_{\Int(\p)}>0$ and ${\ttf}_1|_{\p}\geq0$.
\end{itemize}
\end{lem}
\begin{proof}
Let $\phi_{r,s}(y;z)=\frac{{\ttgg}_2'(y;z)}{{\ttgg}_1'(y;z)}$, where ${\ttgg}_1',{\ttgg}_2'\in\R[\y;\z]$, be a rational separator for $(r,s)$. Recall that $\phi_{r,s}$ is regular on $\Int(\Qq_{r,s})$ and extends to a continuous semialgebraic function $\Phi_{r,s}$ on $\Qq_{r,s}$. As $\Qq_{r,s}$ is by definition a convex polyhedron and $\phi_{r,s}$ is regular on $\Int(\Qq_{r,s})$, we may assume that ${\ttgg}_1'$ has constant sign on $\Int(\Qq_{r,s})$. Write ${\ttgg}_1:=({\ttgg}_1')^2$ and ${\ttgg}_2:={\ttgg}_2'{\ttgg}_1'$. Observe that $\{{\ttgg}_1=0\}\subset\{{\ttgg}_2=0\}$, $\phi_{r,s}=\frac{{\ttgg}_2(y;z)}{{\ttgg}_1(y;z)}$ and ${\ttgg}_1$ is strictly positive on $\Int(\Qq_{r,s})$. Define
$$
{\tt P}(\x):={\ttf}_1(\x')\x_n-{\ttf}_2(\x')\quad\text{where}\quad{\ttf}_\ell(\x')={\ttgg}_\ell({\tt a}_1(\x'),\cdots,{\tt a}_r(\x');{\tt b}_1(\x'),\cdots,{\tt b}_s(\x'))
$$
for $\ell=1,2$. Note that $\{{\ttf}_1=0\}\subset\{{\ttf}_2=0\}$. As $\phi_{r,s}$ is a rational separator for $(r,s)$ and $\Int(\p)\subset\pi_n(\Int(\pol_1))$, we deduce: \em if $x'\in\Int(\p)$, then
$$
\max\{{\tt a}_1(x'),\cdots,{\tt a}_r(x')\}<\frac{{\ttf}_2(x')}{{\ttf}_1(x')}<\min\{{\tt b}_1(x'),\cdots,{\tt b}_s(x')\}.
$$
Besides, ${\ttf}_1(x'):=({\ttgg}_1')^2({\tt a}_1(x'),\cdots,{\tt a}_r(x');{\tt b}_1(x'),\cdots,{\tt b}_s(x'))>0$ for each $x'\in\Int(\p)$\em.

Now, if $x:=(x',x_n)\in\Cc^-$, we have $x'\in\Int(\p)$ and there exists some $i=1,\ldots,r$ such that $x_n<{\tt a}_i(x')$, so ${\tt P}(x)<0$. Similarly, if $x:=(x',x_n)\in\Cc^+$, we have $x'\in\Int(\p)$ and there exists some $j=1,\ldots,s$ such that ${\tt b}_j(x')<x_n$, so ${\tt P}(x)>0$, as required.
\end{proof}

Some technicalities arising from the proof of Theorem \ref{main2} force us to find a kind of analogous result to Lemma \ref{prop:polP} when we are dealing with non-degenerate convex polyhedra not placed in $\vec{\ell}_n$-bounded position.

\subsubsection{Non-degenerate convex polyhedra not placed in $\vec{\ell}_n$-bounded position}\label{equations2}

Let $\pol\subset\R^n$ be a non-degenerate $n$-dimensional convex polyhedron not placed in $\vec{\ell}_n$-bounded position. Let $\Ff_1,\ldots,\Ff_m$ be the facets of $\pol$ and let $H_i:=\{{\tth}_i=0\}$ be the hyperplane generated by $\Ff_i$. Assume that $\pol=\bigcap_{i=1}^mH_i^+$. We may write
\begin{align*}
&{\tth}_i(\x)={\tt b}_i(\x')+\veps_i\x_n\quad\forall\,i=1,\ldots,s,\\
&{\tth}_{s+k}(\x)={\tt c}_k'(\x')\quad\forall\,k=1,\ldots,m-s,
\end{align*}
where $\veps_i=\pm1$. As $\pol$ is non-degenerate, $s\geq1$ and as $\pol\subset\R^n$ is not placed in $\vec{\ell}_n$-bounded position, all $\veps_i$ are either equal to $1$ or $-1$ (that is, all non-vertical facets constitute either the floor or the ceiling of $\pol$). Assume that all $\veps_i=-1$. Observe that $\pol=\pol_1\cap\pol_2$ where
\begin{align*}
&\pol_1:=\{(x',x_n)\in\R^n:\ x_n\leq\min\{{\tt b}_1(x'),\ldots,{\tt b}_s(x')\}\},\\
&\pol_2:=\{(x',x_n)\in\R^n:\ {\tt c}_1(x')\geq0,\ldots,{\tt c}_{m-s}(x')\geq0\}.
\end{align*}
In addition
\begin{align*}
&\Int(\pol_1)=\{(x',x_n)\in\R^n:\ x_n<\min\{{\tt b}_1(x'),\ldots,{\tt b}_s(x')\}\},\\
&\Int(\pol_2)=\{(x',x_n)\in\R^n:\ {\tt c}_1(x')>0,\ldots,{\tt c}_{m-s}(x')>0\}.
\end{align*}

Notice that $\pi_n(\pol_1)=\R^{n-1}$ and
$$
\pi_n(\pol_2)=\{x'\in\R^{n-1}:\ {\tt c}_1(x')\geq0,\ldots,{\tt c}_{m-s}(x')\geq0\}.
$$
As $\pi_n^{-1}(\pi_n(\pol_2))=\pol_2$, we have $\p:=\pi_n(\pol)=\pi_n(\pol_1)\cap\pi_n(\pol_2)=\pi_n(\pol_2)$. Thus, as $\pi_n(\Int(\pol))=\pi_n(\Int(\pol_1))\cap\pi_n(\Int(\pol_2))=\pi_n(\Int(\pol_2))$,
\begin{multline*}
\Int(\p)=\Int(\pi_n(\pol))\\
=\{x'\in\R^{n-1}:\ {\tt c}_1(x')>0,\ldots,{\tt c}_{m-s}(x')>0\}=\pi_n(\Int(\pol_2))=\pi_n(\Int(\pol)).
\end{multline*}
In particular, $\pol_2=\p\times\R$ and $\Int(\pol_2)=\Int(\p)\times\R$. Notice that: \em the semialgebraic set 
$$
\Cc^+:=(\Int(\p)\times\R)\setminus\pol=\{(x',x_n)\in\pol_2:\ \min\{{\tt b}_1(x'),\ldots,{\tt b}_s(x')\}<\,x_n\}
$$ 
is connected\em.

\begin{lem}\label{prop:polP2}
There exists a polynomial ${\tt P}(\x):={\ttf}_1(\x')\x_n-{\ttf}_2(\x')\in\R[\x',\x_n]$ such that ${\ttf}_1(\x')=1$, $-{\ttf}_2>0$ on $\R^{n-1}$ and ${\tt P}|_{\Cc^+}>0$.
\end{lem}
\begin{proof}
Let ${\ttf}_2(\x'):=-\sum_{j=1}^s\frac{{\tt b}_j^2(\x')+1}{2}$ and ${\tt P}(\x):=\x_n-{\ttf}_2(\x')$. Observe that if $x=(x',x_n)\in\Cc^+:=(\Int(\p)\times\R)\setminus\pol$, then $x_n-\min\{{\tt b}_1(x'),\ldots,{\tt b}_s(x')\}>0$. Thus
$$
-x_n<-\min\{{\tt b}_1(x'),\ldots,{\tt b}_s(x')\}=\max\{-{\tt b}_1(x'),\ldots,-{\tt b}_s(x')\}<\sum_{j=1}^s\frac{{\tt b}_j^2(x')+1}{2}
$$
for each $x'\in\R^{n-1}$, so ${\tt P}(x',x_n)>0$ for each $(x',x_n)\in\Cc^+$.
\end{proof}

\section{Complements of convex polyhedra}\label{s5}

In this section we prove constructively Theorem \ref{main2}. We develop first in~\ref{ss:tools} some basic tools. In this regard, Lemma~\ref{prop:face} describes the properties of polynomial maps ${\tt T}_\pol$ (see \eqref{T}), which involve the rational separators. The ideal situation would be that, given a convex polyhedron $\pol$ and one of its facets $\Ff_i$, a polynomial map ${\tt T}_{\pol_{i,\times}}$ would map $\R^n\setminus\pol_{i,\times}$ onto $\R^n\setminus\pol$ (see~\ref{pldrfc} for the definition of $\pol_{i,\times}$). This would allow a neat inductive proof of Theorem~\ref{main2}. However, the pole set of the involved rational separator produces some difficulties. To take care of them we introduce Corollaries~\ref{cor:convex1}, \ref{cor:convex2} and \ref{cor:convex3}. We next consider the complement $\R^n\setminus\Gg_0$ of the interior $\Gg_0$ of a convex polyhedron that contains $\pol$ tightly (in a sense to be described in~\ref{ss:teor4}). This complement is by Theorem~\ref{main1} a polynomial image of $\R^n$. Then in Theorem~\ref{nd} we apply a sequence of polynomial maps of type~\eqref{T} whose images progressively fill the remaining gap $\Gg_0\setminus\pol$, until we finally accomplish to represent $\R^n\setminus\pol$ as the image of a finite composition of polynomial maps. We hope this brief explanation will soften the technicalities of the process.

\subsection{Tools for the inductive process}\label{ss:tools}
Let $\pol\subset\R^n$ be a non-degenerate $n$-dimensional convex polyhedron and let $\p:=\pi_n(\pol)$. Assume that $\vec{e}_n\not\in\conv{\pol}{}$. Let 
\begin{equation}\label{P}
{\tt P}(\x',\x_n):={\ttf}_1(\x')\x_n-{\ttf}_2(\x')\in\R[\x',\x_n]
\end{equation}
be a polynomial satisfying the conditions of Lemma \ref{prop:polP} if $\pol$ is placed in $\vec{\ell}_n$-bounded position and the conditions of Lemma \ref{prop:polP2} otherwise. Write $\pol:=\bigcap_{i=1}^m\{{\tth}_i\geq0\}$ (minimal presentation) where each ${\tth}_i$ is a linear equation. We may assume that the coefficient of $\x_n$ is non-zero for ${\tth}_i$ if and only if $i=1,\ldots,d\leq m$. As $\pol$ is non-degenerate, $d\geq1$. Consider the polynomial
\begin{equation}\label{L}
{\ttL}:=\prod_{i=1}^m{\tth}_i\in\R[\x]
\end{equation}
and the polynomial map
\begin{equation}\label{T}
{\tt T}_\pol:\R^n\to\R^n,\ x:=(x',x_n)=(x_1,\ldots,x_{n-1},x_n)\mapsto(x',x_n-x_{n-1}{\ttL}^2(x){\tt P}(x)).
\end{equation}
Note that if $H_i:=\{{\tth}_i=0\}$ then ${\tt T}_\pol|_{H_i}=\id_{H_i}$ because ${\ttL}|_{H_i}\equiv0$. Fix $a':=(a_1,\ldots,a_{n-1})\in\R^{n-1}$ and denote the vertical line through the point $(a',0)$ with $\ell_{a'}:=(a',0)+\vec{\ell}_n$.

\paragraph{}\label{deg}We claim: \em ${\tt T}_\pol(\ell_{a'})=\ell_{a'}$ for each $a'\in\R^{n-1}$\em. To that end we prove: \em The polynomial ${\tt Q}_{a'}(\t):=\t-a_{n-1}{\ttL}^2(a',\t){\tt P}(a',\t)$ has odd degree for each $a'\in\R^{n-1}$\em. In addition \em ${\tt Q}_{a'}(\t)=\t$ if and only if either $a_{n-1}{\ttL}^2(a',\t)\equiv0$ or ${\ttf}_1(a')=0$\em.
\begin{proof}
If $\pol$ is placed in $\vec{\ell}_n$-bounded position and ${\ttf}_1(a')=0$, then by Lemma \ref{prop:polP} ${\ttf}_2(a')=0$, so ${\tt P}(a',\t)\equiv0$ and ${\tt Q}_{a'}(\t)=\t$ has odd degree. If $\pol$ is not placed in $\vec{\ell}_n$-bounded position, ${\tt P}(a',\t)$ is a monic polynomial of degree $1$. Therefore we may assume that ${\ttf}_1(a')\neq0$ and ${\tt P}(a',\t)$ is a polynomial of degree $1$. As $\pol$ is non-degenerate, we have that ${\ttL}^2(a',\t)$ is either identically zero or a polynomial of positive degree $2d>0$. Therefore, ${\tt Q}_{a'}(\t)$ is either a polynomial of degree $1$ (if $a_{n-1}{\ttL}^2(a',\t)\equiv0$) or of degree $2d+1>1$ (otherwise).
\end{proof}

Let us analyze the behavior of ${\tt T}_\pol$ over certain subsets of the line $\ell_{a'}$ attending to the position of the latter with respect to $\pol$.

\begin{lem}\label{prop:face}
Let $\pol\subset\R^n$ be an $n$-dimensional non-degenerate convex polyhedron and let $\p:=\pi_n(\pol)$. Let $\Gg,\Rr$ be sets such that $\Int(\pol)\subset\Gg\subset\pol$ and $\Rr\subset\partial\pol$. Given $a':=(a_1,\ldots,a_{n-1})\in\R^{n-1}$, we have:
\begin{itemize}
\item[(i)] If $a'\not\in\p$, then $\ell_{a'}\setminus\pol=\ell_{a'}$ and ${\tt T}_\pol(\ell_{a'}\setminus\pol)={\tt T}_\pol(\ell_{a'})=\ell_{a'}$.
\item[(ii)] If $a'\in\p$ and $a_{n-1}\leq0$, then ${\tt T}_\pol(\ell_{a'}\setminus\Gg)=\ell_{a'}\setminus\Gg$.
\item[(iii)] If $a'\in\Int(\p)$ and $a_{n-1}>0$, then ${\tt T}_\pol(\ell_{a'}\setminus\Rr)=\ell_{a'}$ and ${\tt T}_\pol(\ell_{a'}\setminus\Gg)=\ell_{a'}$ if $\ell_{a'}\cap\Gg$ is a bounded set.
\item[(iv)] If $a'\in\partial\p$ and $a_{n-1}>0$, then 
$$
{\tt T}_\pol(\ell_{a'}\setminus\Rr)=\begin{cases}
\ell_{a'}\setminus\Rr&\text{if ${\ttf}_1(a')=0$ or ${\ttL}(a',{\tt t})\equiv0$},\\
\ell_{a'}&\text{if ${\ttf}_1(a')\neq0$ and ${\ttL}(a',{\tt t})\not\equiv0$}.
\end{cases}
$$
In particular, ${\tt T}_\pol(\ell_{a'}\setminus\Rr)=\ell_{a'}\setminus\Rr$ if $\ell_{a'}\cap\Rr$ contains at least two points.
\end{itemize}
\end{lem}
\begin{proof}
(i) As $\ell_{a'}\setminus\pol=\ell_{a'}$, if follows from \ref{deg} that ${\tt T}_\pol(\ell_{a'}\setminus\pol)={\tt T}_\pol(\ell_{a'})=\ell_{a'}$.

We prove next statements (ii), (iii) and (iv). For all of them $a'\in\p$. Then the set $\pol\cap\ell_{a'}$ is either a segment (which can reduce to a point) or a ray. The end point(s) of $\pol\cap\ell_{a'}$ belong(s) to $\bigcup_{i=1}^d\{{\tth}_i=0\}$. If $\pol\cap\ell_{a'}$ is a segment, we denote the end points of $\pol\cap\ell_{a'}$ with $p^-:=(a',t^-)$ and $p^+:=(a',t^+)$, where $t^-\le t^+$. If $\pol\cap\ell_{a'}$ is a ray, we denote its end point with $p^+:=(a',t^+)$. To avoid unnecessary repetitions, we abuse notation and set $\Cc^-=\varnothing$ and $p^-:=(a',t^-)$ where $t^-:={-\infty}$ if $\pol\cap\ell_{a'}$ is a ray. Consequently,
$$
\ell_{a'}=(\ell_{a'}\cap\Cc^-)\cup\{p^-\}
\cup(\ell_{a'}\cap\Int(\pol))
\cup\{p^+\}\cup(\ell_{a'}\cap\Cc^+).
$$

\begin{center}
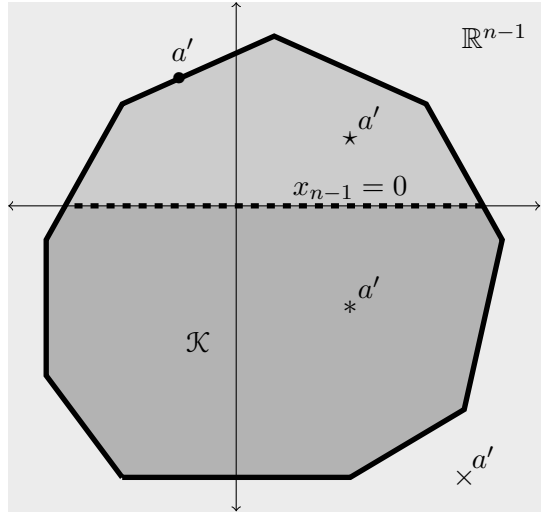

\begin{tikzpicture}[xscale=1, yscale=0.9]

\draw[fill=gray!15,draw=none] (-0.5,0) to (6.5,0) to (6.5,7.5) to (-0.5,7.5) to (-0.5,0);

\draw[fill=gray!60,draw=none] (1,0.5) to (4,0.5) to (5.5,1.5) to (6,4) to (5.75,4.5) to (0.25,4.5) to (0,4) to (0,2) to (1,0.5);
\draw[fill=gray!40,draw=none] (5.75,4.5) to (5,6) to (3,7) to (1,6) to (0.25,4.5);

\draw[line width=2pt] (1,0.5) to (4,0.5) to (5.5,1.5) to (6,4) to (5,6) to (3,7) to (1,6) to (0,4) to (0,2) to (1,0.5);

\draw[line width=2pt,dashed] (5.75,4.5) to (0.25,4.5); 
\draw[<->] (-0.5,4.5) -- (6.5,4.5);
\draw[<->] (2.5,0) -- (2.5,7.5);

\draw (5.5,0.5) node{$\times$};
\draw (4,3) node{$\ast$};
\draw (4,5.5) node{$\star$};
\draw (1.75,6.375) node{$\bullet$};

\draw (5.75,0.8) node{$a'$};
\draw (4.25,3.3) node{$a'$};
\draw (4.25,5.8) node{$a'$};
\draw (1.8,6.8) node{$a'$};
\draw (2,2.5) node{$\pol$};
\draw (5.9,7) node{$\R^{n-1}$};
\draw (4,4.75) node{$x_{n-1}=0$};

\end{tikzpicture}
\captionof{figure}{Analyzed positions (i)--(iv) for $a'\in\R^{n-1}$}
\end{center}
We can write the differences $\ell_{a'}\setminus\Gg$ and $\ell_{a'}\setminus\Rr$ as follows:
\begin{align*}
&\ell_{a'}\setminus\Gg=(\ell_{a'}\cap\Cc^-)\cup(\ell_{a'}\cap\Cc^+)\cup{\mathfrak F}_{\Gg}\\
&\ell_{a'}\setminus\Rr=(\ell_{a'}\cap\Cc^-)\cup(\ell_{a'}\cap\Int(\pol))\cup(\ell_{a'}\cap\Cc^+)\cup{\mathfrak F}_{\Rr}
\end{align*} 
for some ${\mathfrak F}_{\Gg},\Ff_\Rr\subset\{p^-,p^+\}$. Observe that
$$
\ell_{a'}\cap\Cc^-=\{(a,t):\ t<t^-\}\quad\text{and}\quad\ell_{a'}\cap\Cc^+:=\{(a,t):\ t>t^+\}.
$$
By Lemma \ref{prop:polP} we have ${{\tt P}|_{\Cc^-}<0}$ and ${\tt P}|_{\Cc^+}>0$. Also the equality ${\ttL}(p^+)=0$ holds, so ${\tt T}_\pol(p^+)=p^+$. Analogously, if $t^->{-\infty}$ then ${\tt T}_\pol(p^-)=p^-$.

(ii) If $a_{n-1}=0$, then ${\tt Q}_{a'}(\t)=\t$, so ${\tt T}_\pol(\ell_{a'}\setminus\Gg)=\ell_{a'}\setminus\Gg$. Assume next that $a_{n-1}<0$. If $t>t^+$, then ${\tt P}(a',t)>0$, so ${\tt Q}_{a'}(t)\geq t>t^+$ and ${\tt T}_\pol(\ell_{a'}\cap\Cc^+)=\ell_{a'}\cap\Cc^+$ (because ${\tt T}_\pol(p^+)=p^+$). Similarly, if $\ell_{a'}\cap\Cc^-\neq\varnothing$ and $t<t^-$ then ${\tt P}(a',t)<0$, so ${\tt Q}_{a'}(t)\leq t<t^-$ and ${\tt T}_\pol(\ell_{a'}\cap\Cc^-)=\ell_{a'}\cap\Cc^-$ (because ${\tt T}_\pol(p^-)=p^-$). Consequently, ${\tt T}_\pol(\ell_{a'}\setminus\Gg)=\ell_{a'}\setminus\Gg$.

(iii) By Lemma \ref{prop:polP} we have ${\ttf}_1(a')>0$, so the univariate polynomial ${\tt P}(a',\t)={\ttf}_1(a')\t-{\ttf}_2(a')$ has degree one and its leading coefficient is positive. As $a'\in\Int(\p)$, the polynomial ${\ttL}^2(a,\t)$ has even positive degree in $\t$ and positive leading coefficient. Consequently, ${\tt Q}_{a'}(\t)$ is a univariate polynomial of odd degree whose leading coefficient is negative, so $\lim_{t\to\pm\infty}{\tt Q}_{a'}(t)=\mp\infty$. 

If $\ell_{a'}\cap\Gg$ is a bounded set, then $p^-\in\partial\pol$. As ${\tt Q}_{a'}(t^-)=t^-$, ${\tt Q}_{a'}(t^+)=t^+$ and $t^-< t^+$ (because $a'\in\Int(\p)$), we deduce
$$
\R={]{-\infty},t^+[}\cup{]t^-,+\infty[}\subset{\tt Q}_{a'}({]t^+,+\infty[})\cup{\tt Q}_{a'}({]{-\infty},t^-[})\subset\R.
$$
Consequently, ${\tt T}_\pol(\ell_{a'}\setminus\Gg)={\tt T}_\pol(\ell_{a'}\setminus\pol)=\ell_{a'}$.

As $a'\in\Int(\p)$, we have that $\ell_{a'}\cap\Rr\subset\ell_{a'}\cap\partial\pol$ consists of at most two points. 

\paragraph{}
If $\ell_{a'}\cap\Rr$ contains two different points, then $\ell_{a'}\cap\Gg$ is a bounded set and ${\tt T}_\pol(\ell_{a'}\setminus\Rr)={\tt T}_\pol(\ell_{a'}\setminus\pol)=\ell_{a'}$, as we have seen above. 

\paragraph{}\label{sing}
If $\ell_{a'}\cap\Rr$ is a singleton $\{(a', t_0)\}$, we observe that ${\ttL}(a',t_0)=0$ and
$$
\frac{\partial}{\partial \t}{\tt Q}_{a'}(\t)=1-a_{n-1}{\ttL}(a',\t)\Big(2\frac{\partial{\ttL}}{\partial\x_n}(a',\t){\tt P}(a',\t)+{\ttL}(a',\t)\frac{\partial{\tt P}}{\partial\x_n}(a',\t)\Big)\quad\leadsto\quad\frac{\partial}{\partial \t}{\tt Q}_{a'}(t_0)=1.
$$
Thus, there exists $\veps>0$ such that ${\tt Q}_{a'}(t_0-\veps)=:s_1<t_0<s_2:={\tt Q}_{a'}(t_0+\veps)$. As $\lim_{t\to\pm\infty}{\tt Q}_{a'}(t)=\mp\infty$, we have
$$
\R={]{-\infty},s_2[}\cup{]s_1,+\infty[}\subset{\tt Q}_{a'}({]t_0-\veps,+\infty[})\cup{\tt Q}_{a'}({]{-\infty},t_0+\veps[})\subset\R.
$$
Consequently, ${\tt T}_\pol(\ell_{a'}\setminus\Rr)=\ell_{a'}$. 

(iv) By \ref{lem:basic} we have $\ell_{a'}\cap\pol\subset\partial\pol$. If ${\ttf}_1(a')=0$, then by Lemma \ref{prop:polP} also ${\ttf}_2(a')=0$ and ${\tt Q}_{a'}(\t)=\t$, so ${\tt T}_\pol(\ell_{a'}\setminus\Rr)=\ell_{a'}\setminus\Rr$. Analogously, if ${\ttL}(a',\t)\equiv0$, then ${\tt Q}_{a'}(\t)=\t$, so ${\tt T}_\pol(\ell_{a'}\setminus\Rr)=\ell_{a'}\setminus\Rr$.

If ${\ttf}_1(a')\neq0$ and ${\ttL}(a',\t)\not\equiv0$, then $\ell_{a'}\cap\pol$ reduces to a point and $\Rr\cap\ell_{a'}$ can be either empty or a singleton. If $\Rr\cap\ell_{a'}=\varnothing$ it is clear by \ref{deg} that ${\tt T}_\pol(\ell_{a'}\setminus\Rr)=\ell_{a'}$. If $\Rr$ is a singleton, it follows by \ref{sing} that ${\tt T}_\pol(\ell_{a'}\setminus\Rr)=\ell_{a'}$.

Next, if $\Rr\cap\ell_{a'}$ contains two different points, then by \ref{lem:basic} the segment connecting these two points is contained in $\ell_{a'}\cap\pol\subset\partial\pol$. Thus, $\Rr$ is contained in a vertical facet of $\pol$. Consequently, ${\ttL}(a',\t)\equiv0$, so ${\tt Q}_{a'}(\t)=\t$ and ${\tt T}_\pol(\ell_{a'}\setminus\Rr)=\ell_{a'}\setminus\Rr$.
\end{proof}

We present next some technical consequences of the previous result, which are useful for the proof of Theorem \ref{main2}. We need to understand how are the images under ${\tt T}_{\pol}$ of certain sets $\Gg$ satisfying $\Int(\pol)\subset\Gg\subset\pol$. From now on we denote the hyperplane $\{\x_{n-1}=0\}$ with $\Pi_{n-1}$ and the half-spaces $\{\x_{n-1}\le 0\}$ and $\{\x_{n-1}\ge 0\}$ with $\Pi_{n-1}^-$ and $\Pi_{n-1}^+$ respectively.

\begin{cor}\label{cor:convex1}
Let $\pol\subset\R^n$ be an $n$-dimensional non-degenerate convex polyhedron placed in $\vec{\ell}_n$-bounded position and let $\p:=\pi_n(\pol)$. Let $\Gg$ be a set such that $\Int(\pol)\subset\Gg\subset\pol$. Let us set $\p:=\pi_n(\pol)$ and $\Gg^{-}:=\Gg\cap\Pi_{n-1}^-$. Then 
\begin{equation}\label{statement}
(\R^n\setminus\Gg^{-})\setminus(\Gg\cap\pi_n^{-1}(\partial\p)\cap\Int(\Pi_{n-1}^+))\subset{\tt T}_\pol(\R^n\setminus\Gg)\subset\R^n\setminus\Gg^{-}
\end{equation}
and ${\tt T}_\pol(\R^n\setminus\Gg^{-})=\R^n\setminus\Gg^{-}$.
\end{cor}
\begin{proof}
Let $a'\in\R^{n-1}$ and let $\ell_{a'}$ be the vertical line through $(a',0)$. If $a_{n-1}\le 0$, we have by Lemma \ref{prop:face} (i) and (ii)
$$
{\tt T}_\pol(\ell_{a'}\setminus\Gg)=\ell_{a'}\setminus\Gg=\ell_{a'}\setminus\Gg^{-}.
$$
Consequently
\begin{equation}\label{parte1}
{\tt T}_\pol((\R^n\setminus\Gg)\cap\Pi_{n-1}^-)=\bigcup_{\substack{a'\in\R^{n-1}\\a_{n-1}\leq0}}{\tt T}_\pol(\ell_{a'}\setminus\Gg)=\bigcup_{\substack{a'\in\R^{n-1}\\a_{n-1}\leq0}}\ell_{a'}\setminus\Gg^{-}=(\R^n\setminus\Gg^{-})\cap\Pi_{n-1}^-.
\end{equation}
If $a_{n-1}>0$, we have by \ref{deg} and Lemma \ref{prop:face},
$$
{\tt T}_\pol(\ell_{a'}\setminus\Gg)=\begin{cases}
\ell_{a'}&\text{if $\ell\cap\Gg=\varnothing$,}\\
\ell_{a'}&\text{if $a'\in\Int\p$ (because $\ell_{a'}\cap\pol$ is bounded),}\\
\ell_{a'}\setminus\Gg&\text{if $a'\in\partial\p$ and $\ell\cap\Gg$ contains at least two points.}
\end{cases}
$$
In case $a'\in\partial\p$ and $\ell\cap\Gg$ is a singleton we have $\ell_{a'}\setminus\Gg\subset{\tt T}_\pol(\ell_{a'}\setminus\Gg)\subset\ell_{a'}$. Thus, 
\begin{equation}\label{cases}
\begin{cases}
{\tt T}_\pol(\ell_{a'}\setminus\Gg)=\ell_{a'}&\text{if $a'\not\in\partial\p,\ a_{n-1}>0$,}\\
\ell_{a'}\setminus\Gg\subset{\tt T}_\pol(\ell_{a'}\setminus\Gg)\subset\ell_{a'}&\text{if $a'\in\partial\p,\ a_{n-1}>0$.}
\end{cases}
\end{equation}
Using the following equality
$$
{\tt T}_\pol((\R^n\setminus\Gg)\cap\Int(\Pi_{n-1}^+))=\bigcup_{\substack{a'\in\R^{n-1}\\a_{n-1}>0}}{\tt T}_\pol(\ell_{a'}\setminus\Gg)
$$
and \eqref{cases} we conclude
\begin{align}\label{parte2}
{\tt T}_\pol((\R^n\setminus\Gg)\cap\Int(\Pi_{n-1}^+))&\subset\bigcup_{\substack{a'\in\R^{n-1}\\a_{n-1}>0}}\ell_{a'}=\Int(\Pi_{n-1}^+)=(\R^n\setminus\Gg^{-})\cap\Int(\Pi_{n-1}^+)\nonumber\\
{\tt T}_\pol((\R^n\setminus\Gg)\cap\Int(\Pi_{n-1}^+))&\supset\bigcup_{\substack{a'\in\R^{n-1}\\a_{n-1}>0,a'\not\in\partial\p}}\ell_{a'}\cup\bigcup_{\substack{a'\in\R^{n-1}\\a_{n-1}>0,a'\in\partial\p}}\ell_{a'}\setminus\Gg\nonumber\\
&=\Int(\Pi_{n-1}^+)\setminus(\Gg\cap\pi_n^{-1}(\partial\p))\\
&=((\R^n\setminus\Gg^{-})\cap\Int(\Pi_{n-1}^+))\setminus(\Gg\cap\pi_n^{-1}(\partial\p)\cap\Int(\Pi_{n-1}^+)).\nonumber
\end{align}
By \eqref{parte1} and \eqref{parte2} it follows \eqref{statement}, as required. The last equality in the statement follows from equation \eqref{parte1} and the fact that ${\tt T}_\pol(\ell_{a'})=\ell_{a'}$ for each $a'\in\R^{n-1}$.
\end {proof}

\begin{lem}\label{lem:face}
Let $\Ee$ be a face of an $n$-dimensional convex polyhedron $\pol\subset\R^n$ and let $p,q\in\Int(\Ee)$. Let $\ell_p$ be a line through $p$ that meets $\Int(\pol)$ and let $\ell_q$ be a line through $q$ and parallel to $\ell_p$. Then $\ell_q$ also meets $\Int(\pol)$.
\end{lem}
\begin{proof}
Write $\pol=\bigcap_{i=1}^r\{{\tth}_i\geq0\}$ where each ${\tth}_i$ is a linear equation and $\ell_p:=\{p+t\vec{v}:\ t\in\R\}$. We may assume $\Int(\Ee)=\bigcap_{i=1}^k\{{\tth}_i=0\}\cap\bigcap_{i=k+1}^r\{{\tth}_i>0\}$, so ${\tth}_i(p)=0$ for $i=1,\ldots,k$ and ${\tth}_i(p)>0$ for $i=k+1,\ldots,r$. As $\ell_p\cap\Int(\pol)\neq\varnothing$, we may assume (changing $\vec{v}$ by $-\vec{v}$ if necessary) that there exists $t>0$ such that ${\tth}_i(p+t\vec{v})>0$ for $i=1,\ldots,r$. Consequently, $\vec{\tth}_i(\vec{v})>0$ for $i=1,\ldots,k$. As ${\tth}_i(q)=0$ for $i=1,\ldots,k$ and ${\tth}_i(q)>0$ for $i=k+1,\ldots,r$, we have ${\tth}_i(q+t'\vec{v})>0$ for $i=1,\ldots,r$ and $t'>0$ small enough. Thus, $\ell_q\cap\Int(\pol)\neq\varnothing$, as required.
\end{proof}

\begin{cor}\label{cor:convex2}
Let $\pol\subset\R^n$ be an $n$-dimensional non-degenerate convex polyhedron and let $\Ee$ be a face of $\pol$ that is non-parallel to $\Pi_{n-1}$ and meets the open half-space $\Int(\Pi_{n-1}^+)$. Let $\Rr_0\subset\partial\pol$ be such that $\Int(\Ee)\cap\Rr_0=\varnothing$ and let $\Gg$ be a set such that $\Int(\pol)\cup\Int(\Ee)\cup\Rr_0\subset\Gg\subset\pol$. Denote $\Gg^{-}:=\Gg\cap\Pi_{n-1}^-$ and $\p:=\pi_n(\pol)$. Then, after an affine change of coordinates that keeps the hyperplane $\Pi_{n-1}$ invariant (and only depends on $\Ee$), we have
\begin{equation}\label{statement2}
(\R^n\setminus\Gg^{-})\setminus(\Rr_0\cap\pi_n^{-1}(\partial\p))\subset{\tt T}_\pol((\R^n\setminus\Gg^{-})\setminus(\Rr_0\cup\Int(\Ee)))\subset\R^n\setminus\Gg^{-}
\end{equation}
and ${\tt T}_\pol(\R^n\setminus\Gg^{-})=\R^n\setminus\Gg^{-}$.
\end{cor}
\begin{proof}
Take a point $p\in\Int(\Ee)$ and let $\Pi_{n-1,p}$ be the hyperplane through $p$ parallel to $\Pi_{n-1}$. As $\Ee$ is not parallel to $\Pi_{n-1}$, the intersection $\Pi_{n-1,p}\cap\Int(\pol)\neq\varnothing$. Otherwise, as $p\in\pol\cap\Pi_{n-1,p}\subset\partial\pol$, we have that $\Pi_{n-1,p}$ is a supporting hyperplane for $\pol$. As $p\in\Int(\Ee)$, we have $\Int(\Ee)\subset\Pi_{n-1,p}$, so $\Ee$ is parallel to $\Pi_{n-1}$, which is a contradiction. 

Let $q\in\Int(\pol)\cap\Pi_{n-1,p}$. After an affine change of coordinates that keeps the hyperplane $\Pi_{n-1}$ invariant, we may assume that the line through $p$ and $q$ is vertical and $\vec{e}_n\not\in\conv{\pol}{}$. This latter condition is possible because as $\pol$ is non-degenerate, either $\vec{e}_n$ or $-\vec{e}_n$ does not belong to $\conv{\pol}{}$. Therefore, $\pi_n(p)\in\Int(\p)=\pi_n(\Int(\pol))$ and by Lemma \ref{lem:face} we have $\pi_n(\Int(\Ee))\subset\Int(\p)$.

The rest of the proof is similar to the one of Corollary \ref{cor:convex1}. We include all the technicalities for the sake of completeness. Observe that 
\begin{equation}\label{ide}
\Gg^{-}=(\Gg^{-}\cup\Rr_0\cup\Int(\Ee))\cap\Pi_{n-1}^-.
\end{equation} 
Let $a'\in\R^{n-1}$ and let $\ell_{a'}$ be the vertical line through $(a',0)$. If $a_{n-1}\le 0$ we have by Lemma \ref{prop:face} (i) and (ii) 
$$
{\tt T}_\pol(\ell_{a'}\setminus\Gg^{-})=\ell_{a'}\setminus\Gg^{-}.
$$
Consequently
\begin{equation}\label{parte12}
{\tt T}_\pol((\R^n\setminus\Gg^{-})\cap\Pi_{n-1}^-)=\bigcup_{\substack{a'\in\R^{n-1}\\a_{n-1}\leq0}}{\tt T}_\pol(\ell_{a'}\setminus\Gg)=\bigcup_{\substack{a'\in\R^{n-1}\\a_{n-1}\leq0}}\ell_{a'}\setminus\Gg^{-}=(\R^n\setminus\Gg^{-})\cap\Pi_{n-1}^-.
\end{equation}
By \eqref{ide} it holds
\begin{equation}\label{parte12a}
{\tt T}_\pol(((\R^n\setminus\Gg^{-})\setminus(\Rr_0\cup\Int(\Ee)))\cap\Pi_{n-1}^-)={\tt T}_\pol((\R^n\setminus\Gg^{-})\cap\Pi_{n-1}^-)=(\R^n\setminus\Gg^{-})\cap\Pi_{n-1}^-.
\end{equation}

Observe also that 
$$
(\Gg^{-}\cup\Rr_0\cup\Int(\Ee))\cap\Int(\Pi_{n-1}^+)=(\Rr_0\cup\Int(\Ee))\cap\Int(\Pi_{n-1}^+)\subset\partial\pol\cap\Int(\Pi_{n-1}^+).
$$
We have shown above that $\pi_n(\Int(\Ee))\subset\Int(\p)$. If $a_{n-1}>0$ we have by Lemma \ref{prop:face}
$$
{\tt T}_\pol(\ell_{a'}\setminus(\Rr_0\cup\Int(\Ee)))=\begin{cases}
\ell_{a'}&\text{if $\ell\cap(\Rr_0\cup\Int(\Ee))=\varnothing$,}\\
\ell_{a'}&\text{if $a'\in\Int\p$ (because $\Rr_0\cup\Int(\Ee)\subset\partial\pol$),}\\
\ell_{a'}\setminus\Rr_0&\text{if $a'\in\partial\p$ and $\ell\cap\Rr_0$ contains at least two points.}
\end{cases}
$$
In case $a'\in\partial\p$ and $\ell\cap\Rr_0$ is a singleton, we have $\ell_{a'}\setminus\Rr_0\subset{\tt T}_\pol(\ell_{a'}\setminus\Rr_0)\subset\ell_{a'}$. Thus, 
\begin{equation}\label{cases2}
\begin{cases}
{\tt T}_\pol(\ell_{a'}\setminus(\Rr_0\cup\Int(\Ee)))=\ell_{a'}&\text{if $a'\not\in\partial\p,\ a_{n-1}>0$,}\\
\ell_{a'}\setminus\Rr_0\subset{\tt T}_\pol(\ell_{a'}\setminus(\Rr_0\cup\Int(\Ee)))\subset\ell_{a'}&\text{if $a'\in\partial\p,\ a_{n-1}>0$.}
\end{cases}
\end{equation}
Using the following equalities
\begin{multline}\label{extra2}
{\tt T}_\pol((\R^n\setminus\Gg^{-})\setminus(\Rr_0\cup\Int(\Ee)))\cap\Int(\Pi_{n-1}^+))\\
={\tt T}_\pol((\R^n\setminus(\Rr_0\cup\Int(\Ee)))\cap\Int(\Pi_{n-1}^+))=\bigcup_{\substack{a'\in\R^{n-1}\\a_{n-1}>0}}{\tt T}_\pol(\ell_{a'}\setminus(\Rr_0\cup\Int(\Ee)))
\end{multline}
and \eqref{cases2} we conclude
\begin{align}\label{parte22}
{\tt T}_\pol((\R^n\setminus(\Rr_0\cup\Int(\Ee)))\cap\Int(\Pi_{n-1}^+))&\subset\bigcup_{\substack{a'\in\R^{n-1}\\a_{n-1}>0}}\ell_{a'}=\Int(\Pi_{n-1}^+)=(\R^n\setminus\Gg^{-})\cap\Int(\Pi_{n-1}^+)\nonumber\\
{\tt T}_\pol((\R^n\setminus(\Rr_0\cup\Int(\Ee)))\cap\Int(\Pi_{n-1}^+))&\supset\bigcup_{\substack{a'\in\R^{n-1}\\a_{n-1}>0,a'\not\in\partial\p}}\ell_{a'}\cup\bigcup_{\substack{a'\in\R^{n-1}\\a_{n-1}>0,a'\in\partial\p}}\ell_{a'}\setminus\Rr_0\nonumber\\
&=(\Int(\Pi_{n-1}^+)\setminus(\Rr_0\cap\pi_n^{-1}(\partial\p))\\
&=(\Int(\Pi_{n-1}^+)\cap(\R^n\setminus\Gg^{-}))\setminus(\Rr_0\cap\pi_n^{-1}(\partial\p)).\nonumber
\end{align}
From \eqref{parte12a}, \eqref{extra2} and \eqref{parte22} we obtain \eqref{statement2}. The last equality in the statement follows from equation \eqref{parte12} and the fact that ${\tt T}_\pol(\ell_{a'})=\ell_{a'}$ for each $a'\in\R^{n-1}$, as required.
\end{proof}

\begin{cor}\label{cor:convex3}
Let $\pol\subset\R^n$ be an $n$-dimensional non-degenerate convex polyhedron and let $\Ee_1,\ldots,\Ee_s$ be finitely many faces of $\pol$ with disjoint interiors such that each one is non-parallel to $\Pi_{n-1}$ and meets the open half-space $\Int(\Pi_{n-1}^+)$. Let $\Rr:=\bigcup_{k=1}^s\Int(\Ee_k)$ and let $\Gg$ be a set such that $\Int(\pol)\cup\Rr\subset\Gg\subset\pol$. Denote $\Gg^{-}:=\Gg\cap\Pi_{n-1}^-$. Then, there exists a polynomial map ${\tt T}:\R^n\to\R^n$ such that ${\tt T}((\R^n\setminus\Gg^{-})\setminus\Rr)=\R^n\setminus\Gg^{-}$ and ${\tt T}(\R^n\setminus\Gg^{-})=\R^n\setminus\Gg^{-}$.
\end{cor}
\begin{proof}
We proceed by induction on $s$. If $s=1$, the statement follows from Corollary \ref{cor:convex2}. So let us assume that $s>1$ and let $\Rr_0:=\bigcup_{k=1}^{s-1}\Int(\Ee_k)$. By induction hypothesis there exists a polynomial map ${\tt T}_0:\R^n\to\R^n$ such that ${\tt T}_0((\R^n\setminus\Gg^{-})\setminus\Rr_0))=\R^n\setminus\Gg^{-}$ and ${\tt T}_0(\R^n\setminus\Gg^{-})=\R^n\setminus\Gg^{-}$. On the other hand by Corollary \ref{cor:convex2} there exists a polynomial map ${\tt T}_1:\R^n\to\R^n$ such that 
$$
(\R^n\setminus\Gg^{-})\setminus\Rr_0\subset{\tt T}_1((\R^n\setminus\Gg^{-})\setminus(\Rr_0\cup\Int(\Ee)))\subset\R^n\setminus\Gg^{-}
$$
and ${\tt T}_1(\R^n\setminus\Gg^{-})=\R^n\setminus\Gg^{-}$. Consider the polynomial map ${\tt T}={\tt T}_0\circ{\tt T}_1:\R^n\to\R^n$. We have
\begin{multline*}
\R^n\setminus\Gg^{-}={\tt T}_0((\R^n\setminus\Gg^{-})\setminus\Rr_0)\subset{\tt T}_0({\tt T}_1((\R^n\setminus\Gg^{-})\setminus(\Rr_0\cup\Int(\Ee))))\\
={\tt T}((\R^n\setminus\Gg^{-})\setminus\Rr)\subset{\tt T}_0(\R^n\setminus\Gg^{-})=\R^n\setminus\Gg^{-},
\end{multline*}
so ${\tt T}((\R^n\setminus\Gg^{-})\setminus\Rr))=\R^n\setminus\Gg^{-}$. Finally, 
$$
{\tt T}(\R^n\setminus\Gg^{-})={\tt T}_0({\tt T}_1(\R^n\setminus\Gg^{-}))={\tt T}_0(\R^n\setminus\Gg^{-})=\R^n\setminus\Gg^{-},
$$
as required.
\end{proof}

\subsection{Proof of Theorem \ref{main2}}\label{ss:teor4}
Let $\pol\subset\R^n$ be an $n$-dimensional non-degenerate convex polyhedron and let $\Ff_1,\ldots,\Ff_m$ be its facets. Let $H_i$ be the hyperplane generated by $\Ff_i$ and let ${\tth}_i$ be a linear equation of $H_i$ such that $\pol=\bigcap_{i=1}^m\{{\tth}_i\geq0\}$. For each $\veps>0$ denote by $H_i(\veps)$ the hyperplane of linear equation ${\tth}_i+\veps=0$. The hyperplanes $H_i$ and $H_i(\veps)$ are parallel and $H_i^+\subset H_i^+(\epsilon)$.\setcounter{paragraph}{0} 

\paragraph{}\label{I}Let $I:=\{i_1,\ldots,i_k,i_{k+1}\}\subset\{1,\ldots,m\}$ be such that
\begin{itemize}
\item[(1)] The vectorial hyperplanes $\vec{H}_{i_1},\dots,\vec{H}_{i_k}$ are linearly independent.
\item[(2)] $W_I:=\bigcap_{j=1}^kH_{i_j}$ is parallel to $H_{i_{k+1}}$
\item[(3)] $W_I\subset\Int(H_{i_{k+1}}^-)$.
\end{itemize}
Define $\delta(I):=\dist(W_I,H_{i_{k+1}}^+)$. Let ${\mathfrak I}$ be the collection of all the subsets $I\subset\{1,\ldots,m\}$ satisfying conditions (1), (2) and (3). Define
$$
\delta:=\min\{\delta(I):\ I\in{\mathfrak I}\}>0.
$$

\paragraph{}\label{k0}Fix $\veps>0$ such that $\dist(H_i^+,H_i(\veps))<\frac{\delta}{2}$ for $i=1,\ldots,m$ and define
$$
\pol_0:=\bigcap_{i=1}^m H_i^+(\epsilon),
$$
Observe that $\pol\subset\Int(\pol_0)$ and notice that $\pol_0$ is an $n$-dimensional non-degenerate convex polyhedron. Define
$$
\Gg_0:=\Int(\pol_0),\quad\Gg_i:=\Gg_{i-1}\cap H_i^+\quad\text{and}\quad\Gg_m=\pol.
$$
Denote $\pol_i:=\cl(\Gg_i)$ and observe that
$$
\Int(\pol_i)\subset\Gg_i=H_1^+\cap\cdots\cap H_i^+\cap\Int(H_{i+1}^+(\epsilon))\cap\cdots\cap\Int(H_m^+(\epsilon))\subset\pol_i.
$$
By Theorems \ref{main-known} and \ref{main1} the semialgebraic set $\R^n\setminus\Gg_0$ is a polynomial image of $\R^n$. Our goal is to make use of Corollaries \ref{cor:convex1} and \ref{cor:convex3} to represent $\R^n\setminus\pol$ as a polynomial image of $\R^n$. We need first the following property of the sets $\Gg_i$.
\begin{lem}\label{lem:openface}
Fix $i=1,\ldots,m$ and let $\Ee$ be a face of $\pol_i:=\cl(\Gg_i)$ that lies in $\Int(H_{i+1}^-)$ and is parallel to $H_{i+1}$. Then $\Ee\cap\Gg_i=\varnothing$.
\end{lem}
\begin{proof}
Let $W$ be the affine subspace generated by $\Ee$. We claim: \em $W\subset H_\ell(\veps)$ for some $\ell=i+1,\ldots,m$\em. Note that if this is the case then $W\cap\Gg_i=\varnothing$ and so $\Ee\cap\Gg_i=\varnothing$. 

To prove our claim, assume that none of the aforementioned inclusions hold, so that $W=\bigcap_{j=1}^kH_{i_j}$ for some $1\leq i_1,\ldots,i_k\leq i$ such that $\vec{H}_{i_1},\ldots,\vec{H}_{i_k}$ are linearly independent. Thus, $W=W_I$ for $I=\{i_1,\ldots,i_k,i+1\}$ following the notation of \ref{I}. We show now that this cannot happen. As $\Ee\subset\Int(H_{i+1}^-)$ and $\Ee$ is parallel to $H_{i+1}$, we have $\dist(W_I,H_{i+1}^+)\geq\delta$. Since $\dist(H_{i+1}^+,H_{i+1}(\veps))<\frac{\delta}{2}$, we deduce $\Ee\subset W_I\subset\Int(H_{i+1}(\veps)^-)$. On the other hand, we must have $\Ee\subset \pol_i\subset H_{i+1}(\veps)^+$, which is a contradiction.
\end{proof}

We prove next Theorem \ref{main2} for a non-degenerate $n$-dimensional convex polyhedron.

\begin{thm}\label{nd}
Let $\pol\subset\R^n$ be a non-degenerate $n$-dimensional convex polyhedron. Then $\R^n\setminus\pol$ is a polynomial image of $\R^n$.
\end{thm}
\begin{proof}
Consider the non-degenerate $n$-dimensional convex polyhedron $\pol_0$ and the semialgebraic sets $\Gg_0,\dots,\Gg_m$ described in \ref{k0}. Place $\pol$ in $\vec{\ell}_n$-bounded position in such a way that the facet $\Ff_1$ is contained in $\Pi_{n-1}$ and $\pol\subset\Pi_{n-1}^-$ (see Lemma \ref{posfacet}). Notice that $\pol_0=\cl(\Gg_0)$ is also in $\vec{\ell}_n$-bounded position. Observe that $H_1=\Pi_{n-1}$, $H_1^+=\Pi_{n-1}^-$ and $\Gg_0\cap H_1^+=\Gg_1$. By Corollary \ref{cor:convex1} applied to $\Gg_0$ there exists a polynomial map ${\tt T}_0:\R^n\to\R^n$ such that 
$$
(\R^n\setminus\Gg_1)\setminus(\Int(H_1^-)\cap\Gg_0\cap\pi_n^{-1}(\partial\p_0))\subset{\tt T}_0(\R^n\setminus\Gg_0)\subset\R^n\setminus\Gg_1
$$
where $\p_0:=\pi_n(\pol_0)$. As $\Gg_0=\Int(\pol_0)$, we have $\Gg_0\cap\pi_n^{-1}(\partial\p_0)=\varnothing$, so ${\tt T}_0(\R^n\setminus\Gg_0)=\Gg_1$. 

Assume by induction that $\R^n\setminus\Gg_i$ is a polynomial image of $\R^n$. Place $\pol_{i+1}$ in $\vec{\ell}_n$-bounded position in such a way that the facet $\Ff_{i+1}$ is contained in $\Pi_{n-1}=H_{i+1}$ and $\pol\subset\Pi_{n-1}^-=H_{i+1}^+$ (see Lemma \ref{posfacet}). Note that $\pol_i=\cl(\Gg_i)$ is also in $\vec{\ell}_n$-bounded position. By Corollary \ref{cor:convex1} applied to $\Gg_i$ we obtain a polynomial map ${\tt T}_{i,0}:\R^n\to\R^n$ with
$$
(\R^n\setminus\Gg_{i+1})\setminus(\Int(H_{i+1}^-)\cap\Gg_i\cap\pi_n^{-1}(\partial\p_i))\subset{\tt T}_{i,0}(\R^n\setminus\Gg_i)\subset\R^n\setminus\Gg_{i+1}
$$
and ${\tt T}_{i,0}(\R^n\setminus\Gg_{i+1})=\R^n\setminus\Gg_{i+1}$ where $\p_i:=\pi_n(\cl(\Gg_i))$. By \ref{lem:basic} the set $\Gg_i\cap\pi_n^{-1}(\partial\p_i)$ can be expressed as a finite, disjoint union of some interiors of faces of $\pol_i=\cl(\Gg_i)$. Thus, there exist finitely many facets $\Ee_1,\ldots,\Ee_r$ of $\pol_i$ such that $\Gg_i\cap\pi_n^{-1}(\partial\p_i)=\Int(\Ee_1)\sqcup\ldots\sqcup\Int(\Ee_r)$. We may assume that only the faces $\Ee_1,\ldots,\Ee_s$ where $0\leq s\leq r$ intersect $\Int(H_{i+1}^-)$. We claim: \em $\Ee_j$ is non-parallel to $H_{i+1}^-$ for $j=1,\ldots,s$\em.

Suppose that $\Ee_j$ is parallel to $H_{i+1}$ for some $j=1,\ldots,s$. As $\Ee_j\cap\Int(H_{i+1}^-)\neq\varnothing$, we deduce $\Ee_j\subset\Int(H_{i+1}^-)$. By Lemma \ref{lem:openface} $\Int(\Ee_j)\subset\Ee_j\cap\Gg_i=\varnothing$, which is a contradiction. 

Define $\Rr_i:=\Int(\Ee_1)\sqcup\ldots\sqcup\Int(\Ee_s)$. By Corollary \ref{cor:convex3} there exists a polynomial map ${\tt T}_{i,1}:\R^n\to\R^n$ such that 
$$
{\tt T}_{i,1}((\R^n\setminus\Gg_{i+1})\setminus\Rr_i)=\R^n\setminus\Gg_{i+1}
$$
and ${\tt T}_{i,1}(\R^n\setminus\Gg_{i+1})=\R^n\setminus\Gg_{i+1}$. Thus, the polynomial map ${\tt T}_i:={\tt T}_{i,1}\circ{\tt T}_{i,0}:\R^n\to\R^n$ satisfies
\begin{multline*}
\R^n\setminus\Gg_{i+1}={\tt T}_{i,1}((\R^n\setminus\Gg_{i+1})\setminus\Rr_i)\subset{\tt T}_{i,1}({\tt T}_{i,0}(\R^n\setminus\Gg_i))\\
={\tt T}_i(\R^n\setminus\Gg_i)\subset{\tt T}_{i,1}(\R^n\setminus\Gg_{i+1})=\R^n\setminus\Gg_{i+1},
\end{multline*}
that is, ${\tt T}_i(\R^n\setminus\Gg_i)=\R^n\setminus\Gg_{i+1}$.

The composition ${\tt T}:={\tt T}_{m-1}\circ\cdots\circ{\tt T}_0$ satisfies ${\tt T}(\R^n\setminus\Int(\pol_0))=\R^n\setminus\pol$ and since $\R^n\setminus\Int(\pol_0)$ is by Theorems \ref{main-known} and \ref{main1} a polynomial image of $\R^n$, we are done.
\end{proof}

\begin{proof}[Proof of Theorem \em\ref{main2}]
If $\pol$ is non-degenerate, the result follows from Theorem \ref{nd}. So we assume that $\pol$ is degenerate. If that is the case, we may write $\pol$ after an affine change of coordinates as $\pol=\p\times\R^{n-k}$ where $\p\subset\R^k$ is a $k$-dimensional non-degenerate convex polyhedron. If $k\geq2$, there exists by Theorem \ref{nd} a polynomial map ${\tt T}_0:\R^k\to\R^k$ such that ${\tt T}_0(\R^k)=\R^k\setminus\p$, so the polynomial map
$$
{\tt T}:\R^k\times\R^{n-k}\to\R^k\times\R^{n-k},\ x:=(x',x'')\mapsto({\tt T}_0(x'),x'')
$$
satisfies ${\tt T}(\R^n)=\R^n\setminus\pol$. 

If $k=0$, then $\pol=\R^n$ and there is nothing to say. If $k=1$, we may assume that $\p=({-\infty},0]$ (because $\pol$ is not a layer). This means that $\pol=\{\x_1\leq0\}$, so $\R^n\setminus\pol=\{\x_1>0\}$. As it is well-known this set is a polynomial image of $\R^n$. Take for instance the polynomial map
$$
{\tt T}:=({\tt T_1},\ldots,{\tt T}_n):\R^n\to\R^n,\ x=(x_1,x_2,x'')\mapsto((x_1x_2-1)^2+x_1^2,x_2(x_1x_2-1),x''),
$$
where $x'':=(x_3,\ldots,x_n)$ and whose image is $\R^n\setminus\pol=\{\x_1>0\}$ (see \cite{fg1}), as required.
\end{proof}

\end{document}